\documentclass[a4paper,draft,leqno,12pt]{article}

\usepackage{amsmath,bm}
\usepackage{amscd}
\usepackage{amssymb}
\usepackage{enumerate}
\usepackage{indentfirst}
\usepackage{latexsym}
\usepackage{multicol}
\usepackage{color}
\usepackage{theorem}

\makeatletter\@addtoreset{equation}{section}\makeatother

\newtheorem{theorem}{Theorem}

\begingroup
\newtheorem{lemma}{Lemma}[section]
\newtheorem{proposition}[lemma]{Proposition}

\endgroup

\newtheorem{definition}{Definition}
{\theorembodyfont{\rmfamily}\newtheorem{remark}{Remark}[section]}

\newenvironment{proof}{\textit{Proof. }}{\hfill$\Box$}

\newcommand{\ep}{\varepsilon}

\newcommand{\ds}{\displaystyle}
\newcommand{\beq}[1]{\begin{equation} \label{#1}\ds}
\newcommand{\eeq}{\end{equation}}
\newcommand{\bml}[1]{\beq{#1} \begin{array}{c}\ds}
\newcommand{\eml}{\end{array}\eeq}
\newcommand{\beqq}{\begin{equation*}\ds}
\newcommand{\eeqq}{\end{equation*}}
\newcommand{\bmll}{\beqq \begin{array}{c}\ds}
\newcommand{\emll}{\end{array}\eeqq}
\renewcommand{\div}{{\rm div}\,}

\newcommand{\abs}[1]{\ensuremath{\left| #1 \right|}}

\def \de{\partial}

\def \ep{\varepsilon}

\def \Ucal{\mathcal{U}}
\def \d{\mathrm{d}}

\newcommand{\R}{\mathbb{R}}

\def \sil{\rightarrow}

\newcommand{\id}{{\rm Id}}
\newcommand{\Sym}{\mathcal{S}}

\begin{document}

\author{Elisabetta Chiodaroli and Ond\v{r}ej Kreml\thanks{The work of O.K. is part of the SCIEX project 11.152.}}
\title{On the energy dissipation rate of solutions to the compressible isentropic Euler system}
\date{}

\maketitle

\centerline{EPFL Lausanne}

\centerline{Station 8, CH-1015 Lausanne, Switzerland}

\bigskip

\centerline{Institut f\"ur Mathematik, Universit\"at Z\"urich}

\centerline{Winterthurerstrasse 190, CH-8057  Z\"urich, Switzerland}

 
\begin{abstract}
In this paper we extend and complement the results in \cite{ChDLKr}
on the well-posedness issue for weak solutions of the compressible isentropic Euler system in $2$ space dimensions
with pressure law $p(\rho)=\rho^\gamma$, $\gamma \geq 1$. First we show that every Riemann problem whose one-dimensional self-similar solution consists of two shocks 
admits also infinitely many two-dimensional admissible bounded weak solutions (not containing vacuum) generated by the method of De Lellis and Sz\'ekelyhidi \cite{dls1}, \cite{dls2}.
Moreover we prove that for some of these Riemann problems and for $1\leq \gamma < 3$ such solutions have greater energy dissipation rate than the self-similar solution emanating from the same Riemann data. We therefore show
that the maximal dissipation criterion proposed by Dafermos in \cite{Da1} does not favour the classical self-similar solutions.
\end{abstract}

\section{Introduction}

We consider the Cauchy problem for the compressible isentropic Euler
system of gas dynamics in two space dimensions, namely
\begin{equation}\label{eq:Euler system}
\left\{\begin{array}{l}
\partial_t \rho + {\rm div}_x (\rho v) \;=\; 0\\
\partial_t (\rho v) + {\rm div}_x \left(\rho v\otimes v \right) + \nabla_x [ p(\rho)]\;=\; 0\\
\rho (\cdot,0)\;=\; \rho^0\\
v (\cdot, 0)\;=\; v^0 \, ,
\end{array}\right.
\end{equation}
where the unknowns are the density $\rho$ and the velocity $v$
and the $3$ scalar equations correspond to statements of balance for mass and linear momentum.
The pressure $p$ is a function of $\rho$ determined from the constitutive thermodynamic
relations of the gas under consideration and it is assumed to satisfy $p'>0$ (this
hypothesis guarantees also the hyperbolicity of the system on the regions where $\rho$ is
positive).
We will work with pressure laws
$p(\rho)=  \rho^\gamma$
with constant $\gamma\geq 1$. 

The initial value problem \eqref{eq:Euler system} does not have, in general, a global classical solution due to the appearance of singularities even starting 
from regular initial data.
On the other hand weak solutions are known to be non-unique.
In the literature several ways to restore uniqueness have been devised: a classical one, the so called ``entropy criterion'' (see \cite{lax})
consists in complementing the system \eqref{eq:Euler system} with an entropy inequality which should be satisfied (in the sense of distributions)
by admissible (or entropy) solutions. For the specific system \eqref{eq:Euler system} in two-space dimensions the only non-trivial entropy is the total energy
$\eta = \rho \varepsilon(\rho)+\rho \frac{\abs{v}^2}{2}$ where $\varepsilon: \R^+\rightarrow \R$ denotes the internal energy and 
is given through the law $p(r)=r^2 \varepsilon'(r)$.
Thus, \textit{admissible} (or \textit{entropy}) solutions of \eqref{eq:Euler system} are weak solutions of \eqref{eq:Euler system} satisfying in the sense of distributions
the following
entropy inequality
\begin{equation} \label{eq:energy inequality}
\de_t \left(\rho \varepsilon(\rho)+\rho
\frac{\abs{v}^2}{2}\right)+\div_x
\left[\left(\rho\varepsilon(\rho)+\rho
\frac{\abs{v}^2}{2}+p(\rho)\right) v \right]
\;\leq\; 0,
\end{equation}
which is rather a weak form of energy balance. 

Recently a lot of attention has been devoted to the effectiveness of the entropy inequality \eqref{eq:energy inequality} as a selection criterion
among bounded weak solutions in more than one space dimension. In particular in \cite{dls2} and \cite{ch} some wild initial data have been constructed for
which \eqref{eq:Euler system} admits infinitely many admissible solutions. Moreover in \cite{ChDLKr}, the authors showed that the entropy criterion
does not single out unique weak solutions even under very strong assumptions on the initial data ($(\rho^0, v^0) \in W^{1, \infty}(\R^2)$).
Such counterexamples to uniqueness of entropy solutions to \eqref{eq:Euler system} have been constructed 
building on a new method originally developed for constructing $L^\infty$ solutions to the incompressible Euler
system by De Lellis and Sz\'ekelyhidi in \cite{dls1}-\cite{dls2} and based
on convex integration techniques and Baire category arguments (see also \cite{dls3} for a more general survey). 
This method was further improved to generate continuous and H\"older continuous solutions of incompressible Euler, see De Lellis and Sz\'ekelyhidi \cite{dls4}-\cite{dls5},
Buckmaster, De Lellis and Sz\'ekelyhidi \cite{bdls} and Daneri \cite{dan}.  
It has also been applied to other systems of PDEs, we refer the reader to Cordoba, Faraco and Gancedo \cite{cfg}, Chiodaroli, Feireisl and Kreml \cite{ChFeKr},
Shvidkoy \cite{shvidkoy} and Sz\'ekelyhidi \cite{sz2}.

\subsection{Entropy rate admissibility criterion} \label{s:entropy}

The series of negative results concerning the entropy criterion for system \eqref{eq:Euler system} motivated 
us to explore other admissibility criteria which could work in favour of uniqueness.
In this paper, we therefore address an alternative criterion which has been proposed by Dafermos in \cite{Da1} under
the name \textit{entropy rate admissibility criterion}.
In order to formulate this criterion for the specific system \eqref{eq:Euler system} we define the \textit{total energy} 
of the solutions $(\rho,v)$ to \eqref{eq:Euler system} as
\begin{equation} \label{eq:energy0}
E[\rho,v](t) = \int_{\R^2} \left(\rho\ep(\rho) + \rho\frac{\abs{v}^2}{2}\right)\d x .
\end{equation}
Let us remark that in Dafermos' terminology $E[\rho,v](t)$ is called ``total entropy'' (see \cite{Da1}). However, since in the context of system \eqref{eq:Euler system}
the physical energy plays the role of the mathematical entropy, it is more convenient to call $E[\rho,v](t)$ total energy.
The right derivative of $E[\rho,v](t)$ defines the \textit{energy dissipation rate} of $(\rho,v)$ at time $t$:
\begin{equation} \label{eq:dissipation rate0}
D[\rho,v](t) = \frac{\d_+ E[\rho,v](t)}{\d t}. 
\end{equation}
We will later on work with solutions with piecewise constant values of $\rho$ and $\abs{v}^2$ and it is easy to see that the total energy of any solution 
we construct is infinite. Therefore we restrict the infinite domain $\R^2$ to a finite box $(-L,L)^2$ and denote
\begin{align}
&E_L[\rho,v](t) = \int_{(-L,L)^2} \left(\rho\ep(\rho) + \rho\frac{\abs{v}^2}{2}\right)\d x \label{eq:energy L}\\
&D_L[\rho,v](t) = \frac{\d_+ E_L[\rho,v](t)}{\d t}. \label{eq:dissipation rate L}
\end{align}
The problem of infinite energy of solutions may be solved also by restricting to a periodic domain
and constructing (locally in time) periodic solutions. We describe this procedure in Section \ref{s:per}.

According to \cite{Da1} we can now define the entropy rate admissibility criterion.
\begin{definition}[Entropy rate admissible solution]\label{d:entropy rate}
A weak solution $(\rho,v)$ of \eqref{eq:Euler system} is called \textit{entropy rate admissible} if 
there exists $L^* > 0$ such that there is no other weak solution $(\overline{\rho},\overline{v})$ with the property that for some $\tau\geq 0$, $(\overline{\rho},\overline{v})(x,t)= (\rho,v)(x,t)$ on $\R^2 \times [0, \tau]$
and $ D_L[\overline{\rho},\overline{v}](\tau) < D_L[\rho,v](\tau) $ for all $L \geq L^*$.
\end{definition}
In other words, we call entropy rate admissible the solution(s) dissipating most total energy.

Dafermos in \cite{Da1} investigates the equivalence of the entropy rate admissibility criterion to other admissibility criteria for hyperbolic conservation laws
in the one-dimensional case: he proves that for a single equation the entropy rate criterion is equivalent to the viscosity criterion in the class
of piecewise smooth solutions; 
moreover he justifies the ``new'' criterion also for the system of two equations which governs the rectilinear isentropic motion of elastic media.
However, Dafermos himself suggests in \cite{Da1} the equations of gas dynamics as another test candidate for the entropy rate criterion.
Further investigation has been carried out by Hsiao in \cite{Hs}. Following the approach of Dafermos, Hsiao proves, in the class
of piecewise smooth solutions, the equivalence of the entropy rate criterion and
the viscosity criterion for the one-dimensional system of equations of nonisentropic gas dynamics in lagrangian formulation with pressure laws $p(\rho)= \rho^\gamma$
for $\gamma\geq 5/3 $ while the same equivalence is disproved for $\gamma < 5/3 $.
For further analysis on the relation between entropy rate minimization and admissibility of solutions for a more general class of evolutionary equations we refer to \cite{Da2}.
However, to our knowledge, up to some time ago the entropy rate criterion had not been tested in the case of several space variables and on broader class of solutions
than the piecewise smooth ones.

Very recently Feireisl in \cite{fe} extended the result of Chiodaroli \cite{ch} in order to obtain
infinitely many admissible weak solutions of \eqref{eq:Euler system} globally in time; as a consequence of his construction he can also prove that
none of these solutions are entropy rate admissible. Even if the result of Feireisl \cite{fe} may suggest the effectiveness of the entropy rate criterion 
to rule out oscillatory solutions constructed by the method of De Lellis and Sz\'ekelyhidi, in this paper we actually show that for specific initial data 
the oscillatory solutions dissipate more energy than the self-similar solution which may be believed to be the physical one.

Our results are also inspired by the work \cite{sz} where Sz\'ekelyhidi
constructed irregular solutions of the incompressible Euler equations with vortex-sheet initial data and computed their dissipation rate.

We focus on the Riemann problem for the system \eqref{eq:Euler system}--\eqref{eq:energy inequality} in two-space dimensions. 
Hence, we denote the space variable as $x=(x_1, x_2)\in \R^2$ and consider initial data in the form
\begin{equation}\label{eq:R_data}
(\rho^0 (x), v^0 (x)) := \left\{
\begin{array}{ll}
(\rho_-, v_-) \quad & \mbox{if $x_2<0$}\\ \\
(\rho_+, v_+) & \mbox{if $x_2>0$,} 
\end{array}\right. 
\end{equation}
where $\rho_\pm, v_\pm$ are constants. 
Our concern has been to compare the energy dissipation rate of standard self-similar solutions associated to the Riemann problem \eqref{eq:Euler system}--\eqref{eq:energy inequality},
\eqref{eq:R_data} with the energy dissipation rate of non-standard solutions for the same problem obtained by the method of De Lellis and Sz\'ekelyhidi. 

We obtained the following results.

\begin{theorem}\label{t:main0}
 Let $p(\rho) = \rho^{\gamma}$ with $\gamma \geq 1$. For every Riemann data \eqref{eq:R_data} such that the self-similar solution 
 to the Riemann problem \eqref{eq:Euler system}--\eqref{eq:energy inequality}, \eqref{eq:R_data} consists of an admissible $1-$shock and an admissible $3-$shock, i.e. $v_{-1} = v_{+1}$ and
 \begin{equation}\label{eq:2shocks condition}
  v_{+2} - v_{-2} < -\sqrt{\frac{(\rho_--\rho_+)(p(\rho_-)-p(\rho_+))}{\rho_-\rho_+}},
 \end{equation}
 there exist infinitely many admissible solutions to \eqref{eq:Euler system}--\eqref{eq:energy inequality}, \eqref{eq:R_data}.
\end{theorem}

Compared to \cite{ChDLKr}, Theorem \ref{t:main0} widely extends the set of initial data for which there exist infinitely many admissible solutions to the Riemann problem. 
Moreover Theorem \ref{t:main0} gives this result for any pressure law $p(\rho) = \rho^\gamma$, 
instead of the specific case $\gamma = 2$ in \cite{ChDLKr}. 
As a consequence of Theorem \ref{t:main0} and by a suitable choice of initial data, we can prove the following main theorem.

\begin{theorem}\label{t:main}
 Let $p(\rho) = \rho^\gamma$, $1 \leq \gamma < 3$. 
 There exist Riemann data \eqref{eq:R_data} for which the self-similar solution to \eqref{eq:Euler system}--\eqref{eq:energy inequality}
 emanating from these data is not entropy rate admissible.
\end{theorem}

Theorem \ref{t:main} ensures that for $1 \leq \gamma < 3$ there exist initial Riemann data \eqref{eq:R_data} for which 
some of the infinitely many nonstandard solutions constructed as in Theorem \ref{t:main0} dissipate more energy than the self-similar solution, 
suggesting in particular that the Dafermos entropy rate admissibility criterion would not pick the self-similar solution as the admissible one.

The paper is organized as follows. In Section \ref{s:SS} we recall the standard theory for the Riemann problem for the compressible isentropic Euler system. 
In Section \ref{s:I} we provide all the necessary definitions and present for completeness the crucial ideas of the method of De Lellis and Sz\'ekelyhidi which
enables the construction of infinitely many bounded weak solutions. In Sections \ref{s:Ex} and \ref{s:Dis} we prove Theorems \ref{t:main0} and \ref{t:main}. 
Finally, Section \ref{s:per} contains concluding remarks.

\section{Self-similar solutions of the Riemann problem}\label{s:SS}

 In this Section we present the classical theory for the Riemann problem for system \eqref{eq:Euler system}--\eqref{eq:energy inequality} with initial data \eqref{eq:R_data}, for more details we refer the reader to the books \cite{br} or \cite{da}. More precisely, we search here for one-dimensional solutions, i.e. functions $\rho(x_2,t)$ and $m(x_2,t) = \rho(x_2,t)v(x_2,t)$ solving the two-dimensional compressible isentropic Euler system, which in this case reads as follows
\begin{equation}\label{eq:explicit system}
\left\{\begin{array}{l}
 \de_t \rho + \de_{x_2} m_2 = 0 \\
 \de_t m_1 + \de_{x_2} \frac{m_1m_2}{\rho} = 0 \\
 \de_t m_2 + \de_{x_2} \left(\frac{m_2^2}{\rho} + p(\rho)\right) = 0
\end{array}\right.
\end{equation}
with initial data
\begin{equation}\label{eq:Rdata m}
(\rho^0 (x), m^0 (x)) := \left\{
\begin{array}{ll}
(\rho_-, m_-) = (\rho_-,\rho_-v_-) \quad & \mbox{if $x_2<0$}\\ \\
(\rho_+, m_+) = (\rho_+,\rho_+v_+)  & \mbox{if $x_2>0$.} 
\end{array}\right. 
\end{equation}

We introduce the state vector $U:= (\rho, m_1, m_2)$ and observe that the system \eqref{eq:explicit system} falls into the class of hyperbolic conservation laws taking the form
$$\de_t U + \de_{x_2} F(U)=0,$$
where 
$$
F(U):=\left( \begin{array}{c}
    m_2 \\
    \frac{m_1 m_2}{\rho}\\
    \frac{{m_2}^2}{\rho}+p(\rho)\\
    \end{array} \right).
$$
The system \eqref{eq:explicit system} is indeed strictly hyperbolic on the part of the state space where $\rho > 0$ (see \cite{da}) since the Jacobian matrix $DF(U)$ has three real distinct eigenvalues 
\begin{equation} \label{eq:eigenvalues}
 \lambda_1= \frac{m_2}{\rho}-\sqrt{p'(\rho)}, \qquad \lambda_2=\frac{m_2}{\rho}, \qquad \lambda_3= \frac{m_2}{\rho}+\sqrt{p'(\rho)}
\end{equation}
and three linearly independent right eigenvectors 
\begin{equation} \label{eq:eigenvectors}
 R_1=\left( \begin{array}{c}
    1\\
    \frac{m_1}{\rho}\\
    \frac{m_2}{\rho}-\sqrt{p'(\rho)}\\
    \end{array} \right), \quad 
R_2=\left( \begin{array}{c}
    0\\
    1\\
    0\\
    \end{array} \right), \quad 
 R_3=\left( \begin{array}{c}
    1\\
    \frac{m_1}{\rho}\\
    \frac{m_2}{\rho}+\sqrt{p'(\rho)}\\
    \end{array} \right).
\end{equation}
The eigenvalues $\lambda_i$ are called the $i$-\textit{characteristic speeds} of the system \eqref{eq:explicit system}. 
Finally, the classical theory yields the existence of three Riemann invariants of the system \eqref{eq:explicit system} (for definitions see \cite{da}).  More precisely, the functions
\begin{equation}\label{eq:invariants}
w_3=\frac{m_2}{\rho} +\int_0^\rho \frac{\sqrt{p'(\tau)}}{\tau} d\tau, \qquad w_2=\frac{m_1}{\rho}, 
\qquad w_1=\frac{m_2}{\rho} -\int_0^\rho \frac{\sqrt{p'(\tau)}}{\tau} d\tau
\end{equation}
are, respectively, ($1$- and $2$-), ($1$- and $3$-), ($2$- and $3$-) Riemann invariants.

We close the introductory remarks by observing that the $2-$characteristic family of the system \eqref{eq:explicit system} is linearly degenerate, i.e. $D\lambda_2\cdot R_2 = 0$, whereas the $1-$ and $3-$characteristic families are genuinely nonlinear. Moreover, the state variable $m_1$ appears only in the second equation of \eqref{eq:explicit system} and thus the system can be decoupled. In particular one can show that if the initial data \eqref{eq:Rdata m} of the Riemann problem satisfy $v_{-1} = v_{+1}$, i.e. $\frac{m_{-1}}{\rho_-} = \frac{m_{+1}}{\rho_+}$, then the first component of the velocity of the self-similar solution has to be equal to this constant, $v_1 = \frac{m_1}{\rho} = v_{-1} = v_{-2}$, for details see \cite[Section 8]{ChDLKr}.

Since this will be the case of all Riemann initial data studied in this paper, we study here further only the reduced system containing only the second component of the momentum (for simplicity of notation we denote it from now on by $m$ instead of $m_2$)
\begin{equation}\label{eq:reduced system}
\left\{\begin{array}{l}
 \de_t \rho + \de_{x_2} m = 0 \\
 \de_t m + \de_{x_2} (\frac{m^2}{\rho} + p(\rho)) = 0
\end{array}\right.
\end{equation}
with initial data\footnote{Here again $m_\pm$ and $v_\pm$ are no longer vectors of two components but scalars.}
\begin{equation}\label{eq:Rdata m2}
(\rho^0 (x), m^0 (x)) := \left\{
\begin{array}{ll}
(\rho_-, m_-) = (\rho_-,\rho_-v_-) \quad & \mbox{if $x_2<0$}\\ \\
(\rho_+, m_+) = (\rho_+,\rho_+v_+)  & \mbox{if $x_2>0$.} 
\end{array}\right. 
\end{equation}

\subsection{Admissible shocks}

In this section we study admissible shocks related to the system \eqref{eq:reduced system}. The Rankine-Hugoniot shock conditions (cf. \cite{da} for relevant definitions)
are as follows. States $(\rho_-,m_-)$ on the left and $(\rho_+,m_+)$ on the right with $\rho_\pm > 0$ are connected with a shock of speed $s \in \mathbb{R}$, $s \neq 0$, if and only if
\begin{align}
 \label{eq:RH1} s(\rho_+ - \rho_-) &= m_+ - m_- \\ 
 \label{eq:RH2} s(m_+ - m-) &= \frac{m_+^2}{\rho_+} - \frac{m_-^2}{\rho_-} + p(\rho_+) - p(\rho_-).
\end{align}
From these equations we easily eliminate the shock speed and achieve
\begin{equation}\label{eq:shock speed}
 s = \pm\sqrt{\frac{\frac{m_+^2}{\rho_+} - \frac{m_-^2}{\rho_-} + p(\rho_+) - p(\rho_-)}{\rho_+ - \rho_-}}.
\end{equation}
Plugging this to \eqref{eq:RH1} and changing the notation from momentum $m_\pm$ to velocity $v_\pm = \frac{m_\pm}{\rho_\pm}$ we get the following useful formula
\begin{equation}\label{eq:useful}
 \rho_+\rho_-(v_+-v_-)^2 = (\rho_+-\rho_-)(p(\rho_+)-p(\rho_-)).
\end{equation}

Let us turn our attention now to the admissibility condition of the shocks: for this purpose we choose the entropy shock admissibility condition (cf. \cite[Section 8.5]{da}), 
since in this paper we work with admissible solutions in the sense of inequality \eqref{eq:energy inequality}. For discussion about various shock admissibility conditions see \cite[Chapter 8]{da}. 

Using the entropy inequality \eqref{eq:energy inequality} for the Euler equations we deduce that the shock is admissible if and only if (again, for convenience we use here notation with $v$ instead of $m$)
\begin{align}
 s\left(\rho_-\ep(\rho_-) - \rho_+\ep(\rho_+) + \frac{\rho_-v_-^2}{2} - \frac{\rho_+v_+^2}{2}\right) & \nonumber \\
  \leq (\rho_-\ep(\rho_-) + p(\rho_-))v_- - (\rho_+\ep(\rho_+) &+ p(\rho_+))v_+ + \frac{\rho_-v_-^3}{2} - \frac{\rho_+v_+^3}{2}.
\end{align}
Using \eqref{eq:RH1} we replace $s$ to get
\begin{align}
 \frac{\rho_-v_- - \rho_+v_+}{\rho_- - \rho_+}\left(\rho_-\ep(\rho_-) - \rho_+\ep(\rho_+) + \frac{\rho_-v_-^2}{2} - \frac{\rho_+v_+^2}{2}\right) & \nonumber \\
  \leq (\rho_-\ep(\rho_-) + p(\rho_-))v_- - (\rho_+\ep(\rho_+) + p(\rho_+))v_+ &+ \frac{\rho_-v_-^3}{2} - \frac{\rho_+v_+^3}{2}.
\end{align}
This yields
\begin{align}
 \frac{\rho_-\rho_+}{\rho_- - \rho_+}\left(-v_-\ep(\rho_+)-v_+\ep(\rho_-)-\frac{v_-v_+^2}{2} - \frac{v_-^2v_+}{2}\right) & \nonumber \\
 \leq (p(\rho_-)v_- - p(\rho_+)v_+) + \frac{\rho_-\rho_+}{\rho_--\rho_+}\Big(-v_-\ep(\rho_-) &- v_+\ep(\rho_+)-\frac{v_+^3}{2} - \frac{v_-^3}{2}\Big)
\end{align}
and further
\begin{equation}
 \frac{\rho_-\rho_+(v_--v_+)^2(v_-+v_+)}{2(\rho_--\rho_+)} \leq (p(\rho_-)v_- - p(\rho_+)v_+) + \frac{\rho_-\rho_+(v_--v_+)(\ep(\rho_+)-\ep(\rho_-))}{\rho_--\rho_+}.
\end{equation}
Using \eqref{eq:useful} on the left hand side we achieve
\begin{equation}
 (v_-+v_+)(p(\rho_-)-p(\rho_+)) \leq 2(p(\rho_-)v_- - p(\rho_+)v_+) + \frac{2\rho_-\rho_+(v_--v_+)(\ep(\rho_+)-\ep(\rho_-))}{\rho_--\rho_+}
\end{equation}
and thus
\begin{equation}\label{eq:admiss simplified}
 (v_+-v_-)\left(p(\rho_-)+p(\rho_+) - 2\rho_-\rho_+\frac{\ep(\rho_-)-\ep(\rho_+)}{\rho_--\rho_+}\right) \leq 0.
\end{equation}

\begin{lemma}\label{l:admiss function}
Let $p(\rho) = \rho^\gamma$ with $\gamma \geq 1$. Then it holds
\begin{equation}\label{eq:admiss expression}
 p(\rho_-)+p(\rho_+) - 2\rho_-\rho_+\frac{\ep(\rho_+)-\ep(\rho_-)}{\rho_+-\rho_-} > 0
\end{equation}
for any $\rho_- \neq \rho_+$, $\rho_\pm > 0$.
\end{lemma}
\begin{proof}
 The relation between $\ep(\rho)$ and $p(\rho)$ is 
 \begin{equation}
  p(\rho) = \rho^2\ep'(\rho)
 \end{equation}
 and therefore for $p(\rho) = \rho^\gamma$ with $\gamma > 1$ we have $\ep(\rho) = \frac{\rho^{\gamma-1}}{\gamma-1}$, whereas for $\gamma = 1$ it is $\ep(\rho) = \log\rho$.
 
 Further observe that, without loss of generality, we may assume that $\rho_+ > \rho_-$. Plugging in the relations for $p(\rho)$ and $\ep(\rho)$ we simplify \eqref{eq:admiss expression} in the case $\gamma > 1$ to
 \begin{equation}\label{eq:expression 2}
  (\gamma-1)(\rho_+^{\gamma+1}-\rho_-^{\gamma+1}) - (\gamma+1)\rho_+\rho_-(\rho_+^{\gamma-1}-\rho_-^{\gamma-1}) > 0.
 \end{equation}
 Dividing \eqref{eq:expression 2} by $\rho_-^{\gamma+1}$ and denoting $z = \frac{\rho_+}{\rho_-}$ it remains to prove that it holds
 \begin{equation}
  f(z) = (\gamma-1)(z^{\gamma+1}-1) - (\gamma+1)z(z^{\gamma-1}-1) > 0
 \end{equation}
 for $z > 1$. However it is not difficult to show that 
 \begin{equation}
  f''(z) = \gamma(\gamma^2-1)z^{\gamma-2}(z-1)
 \end{equation}
 and since $f(1) = f'(1) = 0$ and $f''(z) > 0$ for $z > 1$ and $\gamma > 1$ we conclude that $f(z)$ is convex and increasing function on interval $(1,\infty)$, in particular $f(z) > 0$ for $z > 1$.
 
 We use the same arguments also in the case $\gamma = 1$ which yields instead of \eqref{eq:expression 2}
\begin{equation}\label{eq:expression 3}
  \rho_+^2-\rho_-^2 - 2\rho_-\rho_+\log\frac{\rho_+}{\rho_-} > 0.
 \end{equation}
 We therefore introduce
\begin{equation}
  f(z) = z^2-2z\log z - 1
 \end{equation}
 and argue again that $f(1) = f'(1) = 0$ and $f''(z) > 0$ for $z > 1$, thus $f(z)$ is increasing on $(1,\infty)$ and in particular $f(z) > 0$ for $z > 1$. The proof is finished.
\end{proof}

Returning to \eqref{eq:admiss simplified} we immediately get the following
\begin{lemma}\label{l:admiss condition}
 Let $p(\rho) = \rho^\gamma$ with $\gamma \geq 1$  and let $\rho_\pm > 0$. The states $(\rho_-,\rho_-v_-)$ on the left and $(\rho_+,\rho_+v_+)$ on the right are connected with an admissible shock if and only if $v_- \geq v_+$ and the Rankine-Hugoniot conditions \eqref{eq:RH1}-\eqref{eq:RH2} are satisfied.
\end{lemma}

\subsection{Characterization of simple waves}

We distinguish two classes of admissible shocks, namely $1-$shocks related to the characteristic speed $\lambda_1 = v - \sqrt{p'(\rho)} $ and $3-$shocks related to the characteristic speed $\lambda_3 = v + \sqrt{p'(\rho)}$ introduced in \eqref{eq:eigenvalues}.

Starting with relation \eqref{eq:useful} and having in mind that admissible shocks have to satisfy $v_- \geq v_+$ we immediately achieve
\begin{equation}
 v_- = v_+ + \sqrt{\frac{(\rho_+-\rho_-)(p(\rho_+)-p(\rho_-))}{\rho_+\rho_-}}.
\end{equation}
Plugging this to the formula for the shock speed \eqref{eq:shock speed} we get after some easy calculations
\begin{equation}
 s = \pm \sqrt{\left(v_+ - \text{sign}(\rho_+-\rho_-)\sqrt{\frac{\rho_-(p(\rho_+)-p(\rho_-))}{\rho_+(\rho_+-\rho_-)}}\right)^2}.
\end{equation}
We therefore conclude, that $1-$shocks have to satisfy $\rho_+ > \rho_-$ and $3-$shocks satisfy $\rho_+ < \rho_-$.

In the case of system \eqref{eq:reduced system} it is easy to characterize the rarefaction waves, since the classical theory yields that every $i-$Riemann invariant is constant along any $i-$rarefaction waves, see \cite[Theorem 7.6.6]{da}.

We can now fully characterize admissible shocks and rarefaction waves, thus all simple $i-$wave curves in the state space, $i = 1,3$.

\begin{lemma}\label{l:simple waves}
 Let $p(\rho) = \rho^\gamma$ with $\gamma \geq 1$ and let $\rho_\pm > 0$. The states $(\rho_-,\rho_-v_-)$ on the left and $(\rho_+,\rho_+v_+)$ on the right are connected with
 \begin{itemize}
  \item admissible $1-$shock if and only if
  \begin{align}
   \rho_+ &> \rho_- \\
   v_- &= v_+ + \sqrt{\frac{(\rho_+-\rho_-)(p(\rho_+)-p(\rho_-))}{\rho_+\rho_-}} \\
   \text{the speed of the shock is } \quad s &= v_+ - \sqrt{\frac{\rho_-(p(\rho_+)-p(\rho_-))}{\rho_+(\rho_+-\rho_-)}}
  \end{align}
  \item admissible $3-$shock if and only if
  \begin{align}
   \rho_+ &< \rho_- \\
   v_- &= v_+ + \sqrt{\frac{(\rho_+-\rho_-)(p(\rho_+)-p(\rho_-))}{\rho_+\rho_-}} \\
   \text{the speed of the shock is } \quad s &= v_+ + \sqrt{\frac{\rho_-(p(\rho_+)-p(\rho_-))}{\rho_+(\rho_+-\rho_-)}}
  \end{align}
  \item $1-$rarefaction wave if and only if 
  \begin{align}
   \rho_+ &< \rho_- \\
   v_- &= v_+ - \int_{\rho_+}^{\rho_-}\frac{\sqrt{p'(\tau)}}{\tau} \d\tau
  \end{align}
  \item $3-$rarefaction wave if and only if 
  \begin{align}
   \rho_+ &> \rho_- \\
   v_- &= v_+ - \int_{\rho_-}^{\rho_+}\frac{\sqrt{p'(\tau)}}{\tau} \d\tau.
  \end{align}
 \end{itemize}
\end{lemma}

\subsection{Solutions to the Riemann problem}

Using Lemma \ref{l:simple waves} we now characterize the types of self-similar solutions to the Riemann problem \eqref{eq:reduced system}--\eqref{eq:Rdata m2}.
\begin{lemma}\label{l:Riemann selfsimilar solutions}
 Let $\rho_\pm,v_\pm$ be given constants, $\rho_\pm > 0$, and let $p(\rho) = \rho^\gamma$ with $\gamma \geq 1$. Assume for simplicity that $(\rho_-,\rho_-v_-)$ and $(\rho_+,\rho_+v_+)$ do not lie on any simple $i-$wave curve (otherwise the form of the self-similar solution is obvious and given directly by Lemma \ref{l:simple waves}).
\begin{itemize}
 \item[1)] If
 \begin{equation}
  v_+ - v_- \geq \int_0^{\rho_-} \frac{\sqrt{p'(\tau)}}{\tau} \d\tau + \int_0^{\rho_+} \frac{\sqrt{p'(\tau)}}{\tau} \d\tau,
 \end{equation}
  then there exists a unique self-similar solution to the Riemann problem \eqref{eq:reduced system}--\eqref{eq:Rdata m2} consisting of a $1-$rarefaction wave and a $3-$rarefaction wave. The intermediate state is vacuum, i.e. $\rho_m = 0$. 
(For detailed analysis of Riemann problems with vacuum see \cite{LiSm}.)
 \item[2)] If
 \begin{equation}
 \abs{\int_{\rho_-}^{\rho_+} \frac{\sqrt{p'(\tau)}}{\tau} \d\tau} < v_+ - v_- < \int_0^{\rho_-} \frac{\sqrt{p'(\tau)}}{\tau} \d\tau + \int_0^{\rho_+} \frac{\sqrt{p'(\tau)}}{\tau} \d\tau,
 \end{equation}
 then there exists a unique self-similar solution to the Riemann problem \eqref{eq:reduced system}--\eqref{eq:Rdata m2} consisting of a $1-$rarefaction wave and a $3-$rarefaction wave. The intermediate state $(\rho_m,\rho_mv_m)$ is given as a unique solution of the system of equations
 \begin{align}
  v_+ - v_- &= \int_{\rho_m}^{\rho_-} \frac{\sqrt{p'(\tau)}}{\tau} \d\tau + \int_{\rho_m}^{\rho_+} \frac{\sqrt{p'(\tau)}}{\tau} \d\tau \\
  v_m &= v_- + \int_{\rho_m}^{\rho_-} \frac{\sqrt{p'(\tau)}}{\tau} \d\tau.
 \end{align}
 \item[3)] If $\rho_- > \rho_+$ and 
 \begin{equation}
 -\sqrt{\frac{(\rho_--\rho_+)(p(\rho_-)-p(\rho_+))}{\rho_-\rho_+}} < v_+ - v_- < \int_{\rho_+}^{\rho_-} \frac{\sqrt{p'(\tau)}}{\tau} \d\tau,
 \end{equation}
 then there exists a unique self-similar solution to the Riemann problem \eqref{eq:reduced system}--\eqref{eq:Rdata m2} consisting of a $1-$rarefaction wave and an admissible $3-$shock. The intermediate state $(\rho_m,\rho_mv_m)$ is given as a unique solution of the system of equations
 \begin{align}
  v_+ - v_- &= \int_{\rho_m}^{\rho_-} \frac{\sqrt{p'(\tau)}}{\tau} \d\tau - \sqrt{\frac{(\rho_m-\rho_+)(p(\rho_m)-p(\rho_+))}{\rho_m\rho_+}} \\
  v_m &= v_- + \int_{\rho_m}^{\rho_-} \frac{\sqrt{p'(\tau)}}{\tau} \d\tau.
 \end{align}
 \item[4)] If $\rho_- < \rho_+$ and 
 \begin{equation}
 -\sqrt{\frac{(\rho_+-\rho_-)(p(\rho_+)-p(\rho_-))}{\rho_+\rho_-}} < v_+ - v_- < \int_{\rho_-}^{\rho_+} \frac{\sqrt{p'(\tau)}}{\tau} \d\tau,
 \end{equation}
 then there exists a unique self-similar solution to the Riemann problem \eqref{eq:reduced system}--\eqref{eq:Rdata m2} consisting of an admissible $1-$shock and a $3-$rarefaction wave. The intermediate state $(\rho_m,\rho_mv_m)$ is given as a unique solution of the system of equations
 \begin{align}
  v_+ - v_- &= \int_{\rho_m}^{\rho_+} \frac{\sqrt{p'(\tau)}}{\tau} \d\tau - \sqrt{\frac{(\rho_m-\rho_-)(p(\rho_m)-p(\rho_-))}{\rho_m\rho_-}} \\
  v_m &= v_- - \sqrt{\frac{(\rho_m-\rho_-)(p(\rho_m)-p(\rho_-))}{\rho_m\rho_-}}.
 \end{align}
 \item[5)] If 
 \begin{equation}
 v_+ - v_- < -\sqrt{\frac{(\rho_+-\rho_-)(p(\rho_+)-p(\rho_-))}{\rho_+\rho_-}}
 \end{equation}
 then there exists a unique self-similar solution to the Riemann problem \eqref{eq:reduced system}--\eqref{eq:Rdata m2} consisting of an admissible $1-$shock and an admissible $3-$shock. The intermediate state $(\rho_m,\rho_mv_m)$ is given as a unique solution of the system of equations
 \begin{align}
  v_+ - v_- &= - \sqrt{\frac{(\rho_m-\rho_-)(p(\rho_m)-p(\rho_-))}{\rho_m\rho_-}} - \sqrt{\frac{(\rho_m-\rho_+)(p(\rho_m)-p(\rho_+))}{\rho_m\rho_+}} \\
  v_m &= v_- - \sqrt{\frac{(\rho_m-\rho_-)(p(\rho_m)-p(\rho_-))}{\rho_m\rho_-}}.
 \end{align}
 \end{itemize}
\end{lemma}
\begin{proof}
 The proof is an easy application of Lemma \ref{l:simple waves}, the uniqueness in the class of self-similar solutions is a consequence of \cite[Proposition 8.1]{ChDLKr}. We refer the reader also to \cite[Chapter 9]{da} for general methods of solving Riemann problems.
\end{proof}

\section{Subsolutions}\label{s:I}

In this section we first provide all necessary definitions for the rest of the paper, then we recall the main ingredients
needed in the proof of Theorem \ref{t:main0} which are inherited by the construction carried out in \cite[Section 3]{ChDLKr}.

\begin{definition}[Weak solution]\label{d:weak}
By a {\em weak solution} of \eqref{eq:Euler system} on $\R^2\times[0,\infty)$ we
mean a pair $(\rho, v)\in L^\infty(\R^2\times [0,\infty))$ such that the following identities 
hold for every test functions $\psi\in C_c^{\infty}(\R^2\times [0, \infty))$,
$\phi\in C_c^{\infty}(\R^2\times [0, \infty))$:
\begin{equation} \label{eq:weak1}
\int_0^\infty \int_{\R^2} \left[\rho\de_t \psi+ \rho v \cdot \nabla_x \psi\right] \d x \d t + \int_{\R^2} \rho^0(x)\psi(x,0) \d x = 0
\end{equation}
\begin{align} \label{eq:weak2}
&\int_0^\infty \int_{\R^2} \left[ \rho v \cdot \de_t \phi+ \rho v \otimes v : \nabla_x \phi +p(\rho) \div_x \phi \right] \d x \d t
+\int_{\R^2} \rho^0(x) v^0(x)\cdot\phi(x,0) \d x=0.
\end{align}
\end{definition}

\begin{definition}[Admissible weak solution]\label{d:admissible}
A bounded weak solution $(\rho, v)$ of \eqref{eq:Euler system} is {\em admissible} if
it satisfies the following inequality for every nonnegative 
test function $\varphi\in C_c^{\infty}(\R^2\times [0,\infty))$:
\begin{align} \label{eq:admissibility condition}
 &\int_0^\infty\int_{\R^2} \left[\left(\rho\varepsilon(\rho)+\rho \frac{\abs{v}^2}{2}\right)\de_t \varphi+\left(\rho\varepsilon(\rho)+\rho
\frac{\abs{v}^2}{2}+p(\rho)\right) v \cdot \nabla_x \varphi \right] \d x \d t\notag \\
&+\int_{\R^2} \left(\rho^0 (x) \varepsilon(\rho^0 (x))+\rho^0 (x)\frac{\abs{v^0 (x)}^2}{2}\right) 
\varphi(x,0) \d x \geq 0 .
\end{align}
\end{definition}

We now introduce the notion of \textit{fan subsolution} as in \cite{ChDLKr}.

\begin{definition}[Fan partition]\label{d:fan}
A {\em fan partition} of $\R^2\times (0, \infty)$ consists of three open sets $P_-, P_1, P_+$
of the following form 
\begin{align}
 P_- &= \{(x,t): t>0 \quad \mbox{and} \quad x_2 < \nu_- t\}\\
 P_1 &= \{(x,t): t>0 \quad \mbox{and} \quad \nu_- t < x_2 < \nu_+ t\}\\
 P_+ &= \{(x,t): t>0 \quad \mbox{and} \quad x_2 > \nu_+ t\},
\end{align}
where $\nu_- < \nu_+$ is an arbitrary couple of real numbers.
\end{definition}

We denote by $\Sym_0^{2\times2}$ the set of all symmetric $2\times2$ matrices with zero trace.

\begin{definition}[Fan subsolution] \label{d:subs}
A {\em fan subsolution} to the compressible Euler equations \eqref{eq:Euler system} with
initial data \eqref{eq:R_data} is a triple 
$(\overline{\rho}, \overline{v}, \overline{u}): \R^2\times 
(0,\infty) \rightarrow (\R^+, \R^2, \Sym_0^{2\times2})$ of piecewise constant functions satisfying
the following requirements.
\begin{itemize}
\item[(i)] There is a fan partition $P_-, P_1, P_+$ of $\R^2\times (0, \infty)$ such that
\[
(\overline{\rho}, \overline{v}, \overline{u})=  
(\rho_-, v_-, u_-) \bm{1}_{P_-}
+ (\rho_1, v_1, u_1) \bm{1}_{P_1}
+ (\rho_+, v_+, u_+) \bm{1}_{P_+}
\]
where $\rho_1, v_1, u_1$ are constants with $\rho_1 >0$ and $u_\pm =
v_\pm\otimes v_\pm - \textstyle{\frac{1}{2}} |v_\pm|^2 \id$;
\item[(ii)] There exists a positive constant $C$ such that
\begin{equation} \label{eq:subsolution 2}
v_1\otimes v_1 - u_1 < \frac{C}{2} \id\, ;
\end{equation}
\item[(iii)] The triple $(\overline{\rho}, \overline{v}, \overline{u})$ solves the following system in the
sense of distributions:
\begin{align}
&\partial_t \overline{\rho} + {\rm div}_x (\overline{\rho} \, \overline{v}) \;=\; 0\label{eq:continuity}\\
&\partial_t (\overline{\rho} \, \overline{v})+{\rm div}_x \left(\overline{\rho} \, \overline{u} 
\right) + \nabla_x \left( p(\overline{\rho})+\frac{1}{2}\left( C \rho_1
\bm{1}_{P_1} + \overline{\rho} |\overline{v}|^2 \bm{1}_{P_+\cup P_-}\right)\right)= 0.\label{eq:momentum}
\end{align}
\end{itemize}
\end{definition}

\begin{definition}[Admissible fan subsolution]\label{d:admiss}
 A fan subsolution $(\overline{\rho}, \overline{v}, \overline{u})$ is said to be {\em admissible}
if it satisfies the following inequality in the sense of distributions
\begin{align} 
&\de_t \left(\overline{\rho} \varepsilon(\overline{\rho})\right)+\div_x
\left[\left(\overline{\rho}\varepsilon(\overline{\rho})+p(\overline{\rho})\right) \overline{v}\right]
 + \de_t \left( \overline{\rho} \frac{|\overline{v}|^2}{2} \bm{1}_{P_+\cup P_-} \right)
+ \div_x \left(\overline{\rho} \frac{|\overline{v}|^2}{2} \overline{v} \bm{1}_{P_+\cup P_-}\right)\nonumber\\
&\qquad\qquad+ \left[\de_t\left(\rho_1 \, \frac{C}{2} \, \bm{1}_{P_1}\right) 
+ \div_x\left(\rho_1 \, \overline{v} \, \frac{C}{2}  \, \bm{1}_{P_1}\right)\right]
\;\leq\; 0\, .\label{eq:admissible subsolution}
\end{align}
\end{definition}

The following Proposition is the key ingredient in the presented theory and is proved in \cite{ChDLKr}. Nevertheless, we recall the main ideas of the proof also here for reader's convenience.

\begin{proposition}\label{p:subs}
Let $p$ be any $C^1$ function and $(\rho_\pm, v_\pm)$ be such that there exists at least one
admissible fan subsolution $(\overline{\rho}, \overline{v}, \overline{u})$ of \eqref{eq:Euler system}
with initial data \eqref{eq:R_data}. Then there are infinitely 
many bounded admissible solutions $(\rho, v)$ to \eqref{eq:Euler system}-\eqref{eq:energy inequality}, \eqref{eq:R_data} such that 
$\rho=\overline{\rho}$ and $\abs{v}^2\bm{1}_{P_1} = C$.
\end{proposition}

Roughly speaking, the infinitely many bounded admissible solutions $(\rho,v)$ are constructed by adding to the subsolution solutions to 
 the linearized pressureless incompressible Euler equations supported in $P_1$. Such solutions are given by the following Lemma, cf. \cite[Lemma 3.7]{ChDLKr}.

\begin{lemma}\label{l:ci}
Let $(\tilde{v}, \tilde{u})\in \R^2\times \Sym_0^{2\times 2}$ and $C_0>0$ be such that $\tilde{v}\otimes \tilde{v}
- \tilde{u} < \frac{C_0}{2} \id$. For any open set $\Omega\subset \R^2\times \R$ there are infinitely many maps
$(\underline{v}, \underline{u}) \in L^\infty (\R^2\times \R , \R^2\times \Sym_0^{2\times 2})$ with the following property
\begin{itemize}
\item[(i)] $\underline{v}$ and $\underline{u}$ vanish identically outside $\Omega$;
\item[(ii)] $\div_x \underline{v} = 0$ and $\partial_t \underline{v} + \div_x \underline{u} = 0$;
\item[(iii)] $ (\tilde{v} + \underline{v})\otimes (\tilde{v} + \underline{v}) - (\tilde{u} + \underline{u}) = \frac{C_0}{2} \id$
a.e. on $\Omega$.
\end{itemize}
\end{lemma}

Proposition \ref{p:subs} is then proved by applying Lemma \ref{l:ci} with $\Omega = P_1$, $(\tilde{v}, \tilde{u}) = (v_1,u_1)$ and $C_0 = C$. It is then a matter of easy computation to check that each couple $(\overline{\rho}, \overline{v} + \underline{v})$ is indeed an admissible weak solution to \eqref{eq:Euler system}--\eqref{eq:energy inequality} with initial data \eqref{eq:R_data}, for details see \cite[Section 3.3]{ChDLKr}.

In the rest of this section we present the main ideas of the proof of Lemma \ref{l:ci}. The whole proof can be found in \cite[Section 4]{ChDLKr}.

\begin{proof}[Lemma \ref{l:ci}]
 We define $X_0$ to be the space of 
$(\underline{v}, \underline{u})\in C^\infty_c (\Omega, \R^2\times
\Sym_0^{2\times 2})$ which satisfy (ii) and the pointwise inequality 
$(\tilde{v} + \underline{v})\otimes (\tilde{v} + \underline{v}) - (\tilde{u} + \underline{u}) < \frac{C_0}{2} \id$.
We then consider the closure of $X_0$ in the $L^\infty$ weak$^\star$ topology and denote it $X$. Since $X$ is a bounded (weakly$^\star$)
closed subset of $L^\infty$, the weak$^\star$ topology is metrizable on $X$, hence we achieve a complete metric space $(X,d)$. 
Observe that any element in $X$ satisfies (i) and (ii). We thus prove that on a residual set (in the sense of Baire category) (iii) holds.

The original idea of De Lellis and Sz\'ekelyhidi (cf. \cite{dls1}) is now to consider the identity map $I$ from $(X,d)$ to $L^\infty(\Omega,\R^2\times\Sym_0^{2\times 2})$ endowed with strong $L^2$ topology and prove that for each point of continuity of $I$ (iii) holds. However, since we consider also unbounded domains $\Omega$, for technical reasons we have to consider family of maps $I_N$, $N \in \mathbb{N}\setminus\{0\}$ as follows: to $(\underline{v}, \underline{u})$ we associate
the corresponding restrictions of these maps to $B_N (0) \times (-N, N)$. We then consider $I_N$ as a map from $(X,d)$ to 
$L^\infty (B_N (0)\times (-N, N), \R^2\times \Sym_0^{2\times 2}$) endowed with the strong
$L^2$ topology. Arguing as in \cite[Lemma 4.5]{dls1} we see that each $I_N$ is a Baire-1 map and hence, from
a classical theorem in Baire category, its points of continuity form a residual set in $X$. The set of points at which all of maps $I_N$ are continuous is therefore also a residual set in $X$. We claim now that for each point of continuity of $I_N$ (iii) holds on $B_N (0)\times (-N, N)$. Proceeding as in \cite[Lemma 4.6]{dls1} we prove this claim and therefore finish the proof of Lemma \ref{l:ci} using the following Proposition \ref{p:inequality} and a contradiction argument. 
\end{proof}

\begin{proposition}\label{p:inequality}
 If $(\underline{v}, \underline{u})\in X_0$, then  
there exists a sequence
$(v_k, u_k)\subset X_0$ converging weakly$^\star$ to $(\underline{v},\underline{u})$ for which 
\begin{equation}\label{eq:seq ineq}
\liminf_k \|\tilde{v} + v_k\|_{L^2 (\Gamma)} \geq 
\|\tilde{v} + \underline{v}\|^2_{L^2 (\Gamma)} + \beta \left(C |\Gamma| - \|\tilde{v} + 
\underline{v}\|^2_{L^2 (\Gamma)}\right)^2\, ,
\end{equation}
where $\Gamma= B_N (0)\times (-N, N)$ and $\beta$ depends only on $\Gamma$.
\end{proposition}

\begin{proof}
The proof of Proposition \ref{p:inequality} is based on two crucial observations.

\begin{itemize}
 \item[(1)] First of them is the existence of the plain-wave like solutions to the system of linear PDE's (ii), i.e. compactly supported solutions taking values in an $\ep-$neighborhood of a certain line segments in $\R^2\times \Sym_0^{2\times 2}$, see \cite[Proposition 4.1]{ChDLKr}.
 \item[(2)] The second observation concerns the geometric properties of the set 
$$
\Ucal = \left\{(a,A) \in \R^2\times \Sym_0^{2\times 2}: a\otimes a - A < \frac{C_0}{2}\id\right\}.
$$
Namely it holds that for each $(a,A) \in \Ucal$ there exists a segment $\sigma = [-p,p] \subset \R^2\times \Sym_0^{2\times 2}$ such that $(a,A) + \sigma \subset \Ucal$ and $\abs{p} \geq c_0(C_0 - a^2)$ with $c_0 > 0$ being a geometric constant, see \cite[Lemma 4.3]{ChDLKr}.
\end{itemize}

Let $(\underline{v}, \underline{u})\in X_0$. Consider any point $(x_0, t_0)\in \Gamma$. It can be easily seen that $(\tilde{v}, \tilde{u}) + 
(\underline{v}, \underline{u})$ takes values in $\mathcal{U}$. Then for 
$(a, A)= (\tilde{v}, \tilde{u}) + (\underline{v} (x_0, t_0), \underline{u} (x_0, t_0))$ we find a segment $\sigma$ as in (2)
and choose $r>0$ so that $(\tilde{v}, \tilde{u}) + (\underline{v} (x,t), \underline{u} (x,t))
+ \sigma \subset \mathcal{U}$ for any $(x,t)\in B_r (x_0) \times (t_0-r, t_0+r)$. Note that there always exists such $r > 0$ since $(\underline{v},\underline{u})$ are smooth functions. Then using (1) we find a solution to (ii) and rescale it to obtain $(v_{x_0, t_0, r}, u_{x_0, t_0, r})$ supported in $B_r (x_0) \times (t_0-r, t_0+r)$. Moreover $(\underline{v}, \underline{u}) +
(v_{x_0, t_0, r}, u_{x_0, t_0, r})\in X_0$ provided $\varepsilon$ in (1) is taken sufficiently small.

The sequence $(v_k,u_k)$ is then constructed in the following way. For $k > 0$ we consider finite number of points $(x_j,t_j) \in \Gamma$ such that the sets $B_r (x_j) \times (t_j-r, t_j+r)$ with $r = \frac{1}{k}$ are pairwise disjoint and define
\[
(v_k, u_k):= (\underline{v}, \underline{u}) + \sum_j (v_{x_j, t_j, r}, u_{x_j, t_j, r}) \, .
\]
As is proved in detail in \cite[Section 4.1]{ChDLKr}, this construction can be done in such a way that \eqref{eq:seq ineq} holds.
\end{proof}

\section{Proof of Theorem \ref{t:main0}}\label{s:Ex}

Theorem \ref{t:main0} is here proved using Proposition \ref{p:subs}, i.e. showing the existence of a fan admissible subsolution with appropriate initial data.

First we recall the set of identities and inequalities which defines the fan admissible subsolution with initial data \eqref{eq:R_data}, see also \cite[Section 5]{ChDLKr}.

We introduce the real numbers 
$\alpha, \beta, \gamma_1, \gamma_2, v_{-1}, v_{-2}, v_{+1}, v_{+2}$ such that
\begin{align} 
v_1 &= (\alpha, \beta),\label{eq:v1}\\
v_- &= (v_{-1}, v_{-2})\\
v_+ &= (v_{+1}, v_{+2})\\
u_1 &=\left( \begin{array}{cc}
    \gamma_1 & \gamma_2 \\
    \gamma_2 & -\gamma_1\\
    \end{array} \right)\, .\label{eq:u1}
\end{align}

Then, Proposition \ref{p:subs} translates into the following set of algebraic identities and inequalities.

\begin{proposition}\label{p:algebra}
Let $P_-, P_1, P_+$ be a fan partition as in Definition \ref{d:fan}. The constants 
$v_1, v_-, v_+, u_1, \rho_-, \rho_+, \rho_1$ as in \eqref{eq:v1}-\eqref{eq:u1} define an admissible
fan subsolution as in Definitions \ref{d:subs}-\ref{d:admiss} if and only if the following
identities and inequalities hold:
\begin{itemize}
\item Rankine-Hugoniot conditions on the left interface:
\begin{align}
&\nu_- (\rho_- - \rho_1) \, =\,  \rho_- v_{-2} -\rho_1  \beta \label{eq:cont_left}  \\
&\nu_- (\rho_- v_{-1}- \rho_1 \alpha) \, = \, \rho_- v_{-1} v_{-2}- \rho_1 \gamma_2  \label{eq:mom_1_left}\\
&\nu_- (\rho_- v_{-2}- \rho_1 \beta) \, = \,  
\rho_- v_{-2}^2 + \rho_1 \gamma_1 +p (\rho_-)-p (\rho_1) - \rho_1 \frac{C}{2}\, ;\label{eq:mom_2_left}
\end{align}
\item Rankine-Hugoniot conditions on the right interface:
\begin{align}
&\nu_+ (\rho_1-\rho_+ ) \, =\,  \rho_1  \beta - \rho_+ v_{+2} \label{eq:cont_right}\\
&\nu_+ (\rho_1 \alpha- \rho_+ v_{+1}) \, = \, \rho_1 \gamma_2 - \rho_+ v_{+1} v_{+2} \label{eq:mom_1_right}\\
&\nu_+ (\rho_1 \beta- \rho_+ v_{+2}) \, = \, - \rho_1 \gamma_1 - \rho_+ v_{+2}^2 +p (\rho_1) -p (\rho_+) 
+ \rho_1 \frac{C}{2}\, ;\label{eq:mom_2_right}
\end{align}
\item Subsolution condition:
\begin{align}
 &\alpha^2 +\beta^2 < C \label{eq:sub_trace}\\
& \left( \frac{C}{2} -{\alpha}^2 +\gamma_1 \right) \left( \frac{C}{2} -{\beta}^2 -\gamma_1 \right) - 
\left( \gamma_2 - \alpha \beta \right)^2 >0\, ;\label{eq:sub_det}
\end{align}
\item Admissibility condition on the left interface:
\begin{align}
& \nu_-(\rho_- \varepsilon(\rho_-)- \rho_1 \varepsilon( \rho_1))+\nu_- 
\left(\rho_- \frac{\abs{v_-}^2}{2}- \rho_1 \frac{C}{2}\right)\nonumber\\
\leq & \left[(\rho_- \varepsilon(\rho_-)+ p(\rho_-)) v_{-2}- 
( \rho_1 \varepsilon( \rho_1)+ p(\rho_1)) \beta \right] 
+ \left( \rho_- v_{-2} \frac{\abs{v_-}^2}{2}- \rho_1 \beta \frac{C}{2}\right)\, ;\label{eq:E_left}
\end{align}
\item Admissibility condition on the right interface:
\begin{align}
&\nu_+(\rho_1 \varepsilon( \rho_1)- \rho_+ \varepsilon(\rho_+))+\nu_+ 
\left( \rho_1 \frac{C}{2}- \rho_+ \frac{\abs{v_+}^2}{2}\right)\nonumber\\
\leq &\left[ ( \rho_1 \varepsilon( \rho_1)+ p(\rho_1)) \beta- (\rho_+ \varepsilon(\rho_+)+ p(\rho_+)) v_{+2}\right] 
+ \left( \rho_1 \beta \frac{C}{2}- \rho_+ v_{+2} \frac{\abs{v_+}^2}{2}\right)\, .\label{eq:E_right}
\end{align}
\end{itemize}
\end{proposition}

Our aim is now to show the solvability of the previous identities and inequalities.

The following easy observation simplifies the set of algebraic identities and inequalities a little bit.
\begin{lemma}\label{l:algebra simplified}
Let $v_{-1} = v_{+1}$. Then $\alpha = v_{-1} = v_{+1}$ and $\gamma_2 = \alpha\beta$.
\end{lemma}
\begin{proof}
Multiplying \eqref{eq:cont_left} by $\alpha$ and subtracting \eqref{eq:mom_1_left} we achieve
\begin{equation}\label{eq:blabla1}
 \rho_1(\alpha\beta-\gamma_2) = (\alpha-v_{-1})\rho_-(v_{-2}-\nu_-) = (\alpha-v_{-1})\rho_1(\beta-\nu_-).
\end{equation}
Similarly multiplying \eqref{eq:cont_right} by $\alpha$ and subtracting \eqref{eq:mom_1_right} we achieve
\begin{equation}\label{eq:blabla2}
 \rho_1(\alpha\beta-\gamma_2) = (\alpha-v_{+1})\rho_+(v_{+2}-\nu_+) = (\alpha-v_{+1})\rho_1(\beta-\nu_+).
\end{equation}
Comparing \eqref{eq:blabla1} and \eqref{eq:blabla2} and using the assumption $v_{-1} = v_{+1}$ we get 
\begin{equation}
 (\alpha-v_{-1})\rho_1(\beta-\nu_-) = (\alpha-v_{-1})\rho_1(\beta-\nu_+).
\end{equation}
Since $\nu_- < \nu_+$ by definition and $\rho_1 > 0$ we conclude that $\alpha = v_{-1} = v_{+1}$ and consequently also $\gamma_2 = \alpha\beta$.
\end{proof}

Thus, assuming we have Riemann data \eqref{eq:R_data} such that $v_{-1} = v_{+1}$ we can simplify the set of identities and inequalities as follows:
\begin{itemize}
\item Rankine-Hugoniot conditions on the left interface:
\begin{align}
&\nu_- (\rho_- - \rho_1) \, =\,  \rho_- v_{-2} -\rho_1  \beta \label{eq:cont_left s}  \\
&\nu_- (\rho_- v_{-2}- \rho_1 \beta) \, = \,  
\rho_- v_{-2}^2 - \rho_1(\frac{C}{2} - \gamma_1) +p (\rho_-)-p (\rho_1) \, ;\label{eq:mom_2_left s}
\end{align}
\item Rankine-Hugoniot conditions on the right interface:
\begin{align}
&\nu_+ (\rho_1-\rho_+ ) \, =\,  \rho_1  \beta - \rho_+ v_{+2} \label{eq:cont_right s}\\
&\nu_+ (\rho_1 \beta- \rho_+ v_{+2}) \, = \, \rho_1 (\frac{C}{2} - \gamma_1) - \rho_+ v_{+2}^2 +p (\rho_1) -p (\rho_+) 
\, ;\label{eq:mom_2_right s}
\end{align}
\item Subsolution condition:
\begin{align}
 &\alpha^2 +\beta^2 < C \label{eq:sub_trace s}\\
& \left( \frac{C}{2} -{\alpha}^2 +\gamma_1 \right) \left( \frac{C}{2} -{\beta}^2 -\gamma_1 \right) >0\, ;\label{eq:sub_det s}
\end{align}
\end{itemize}
with admissibility conditions \eqref{eq:E_left} and \eqref{eq:E_right} same as above and $\alpha = v_{-1} = v_{+1}$. Investigating further we make the following observation
\begin{lemma}\label{l:algebra simplified 2}
 A necessary condition for \eqref{eq:sub_trace s}-\eqref{eq:sub_det s} to be satisfied is $\frac{C}{2} -\gamma_1 > \beta^2$. 
\end{lemma}
\begin{proof}
We rewrite \eqref{eq:sub_det s} as 
\begin{equation}
 \left(C - {\alpha}^2 -\left(\frac{C}{2} - \gamma_1\right) \right) \left( \frac{C}{2} -\gamma_1 - {\beta}^2 \right) >0.
\end{equation}
Let us assume that $\frac{C}{2} -\gamma_1 < \beta^2$. Then it has to hold 
\begin{equation}\label{eq:blabla3}
 \left(C - {\alpha}^2 -\left(\frac{C}{2} - \gamma_1\right) \right) < 0,
\end{equation}
however using \eqref{eq:sub_trace s} we get $C > \alpha^2 + \beta^2 > \alpha^2 + \frac{C}{2} - \gamma_1$ which contradicts \eqref{eq:blabla3}.
\end{proof}

This motivates us to introduce new unknowns $0 < \ep_1 = \frac{C}{2} - \gamma_1 - \beta^2$ and $0 < \ep_2 = C-\alpha^2-\beta^2-\ep_1$. We also rewrite the admissibility inequalities as described in the following Lemma.
\begin{lemma}
 In the case $v_{-1} = v_{+1} = \alpha$ and with notation $\ep_1$, $\ep_2$ introduced above, the set of algebraic identities and inequalities \eqref{eq:cont_left s}-\eqref{eq:sub_det s} together with \eqref{eq:E_left}-\eqref{eq:E_right} is equivalent to 
\begin{itemize}
\item Rankine-Hugoniot conditions on the left interface:
\begin{align}
&\nu_- (\rho_- - \rho_1) \, =\,  \rho_- v_{-2} -\rho_1  \beta \label{eq:cont_left ss}  \\
&\nu_- (\rho_- v_{-2}- \rho_1 \beta) \, = \,  
\rho_- v_{-2}^2 - \rho_1(\beta^2 + \ep_1) +p (\rho_-)-p (\rho_1) \, ;\label{eq:mom_2_left ss}
\end{align}
\item Rankine-Hugoniot conditions on the right interface:
\begin{align}
&\nu_+ (\rho_1-\rho_+ ) \, =\,  \rho_1  \beta - \rho_+ v_{+2} \label{eq:cont_right ss}\\
&\nu_+ (\rho_1 \beta- \rho_+ v_{+2}) \, = \, \rho_1 (\beta^2 + \ep_1) - \rho_+ v_{+2}^2 +p (\rho_1) -p (\rho_+) 
\, ;\label{eq:mom_2_right ss}
\end{align}
\item Subsolution condition:
\begin{align}
& \ep_1 > 0 \label{eq:sub_1 ss}\\
& \ep_2 > 0\, ;\label{eq:sub_2 ss}
\end{align}
\item Admissibility condition on the left interface:
\begin{align}
&(\beta-v_{-2})\left(p(\rho_-)+p(\rho_1)-2\rho_-\rho_1\frac{\ep(\rho_-)-\ep(\rho_1)}{\rho_--\rho_1}\right) \nonumber\\
\leq &\ep_1\rho_1(v_{-2}+\beta) - (\ep_1+\ep_2)\frac{\rho_-\rho_1(\beta-v_{-2})}{\rho_--\rho_1}\, ;\label{eq:E_left ss}
\end{align}
\item Admissibility condition on the right interface:
\begin{align}
&(v_{+2}-\beta)\left(p(\rho_1)+p(\rho_+)-2\rho_1\rho_+\frac{\ep(\rho_1)-\ep(\rho_+)}{\rho_1-\rho_+}\right) \nonumber\\
\leq &-\ep_1\rho_1(v_{+2}+\beta) + (\ep_1+\ep_2)\frac{\rho_1\rho_+(v_{+2}-\beta)}{\rho_1-\rho_+}\, .\label{eq:E_right ss}
\end{align}
\end{itemize}
\end{lemma}

\begin{proof}
 The only nontrivial part of the proof is the reformulation of the admissibility conditions. We show the procedure on the admissibility condition on the left interface, the condition on the right interface is achieved in the same way. First we observe that we can subtract from \eqref{eq:E_left} the identity \eqref{eq:cont_left ss} multiplied by $\frac{\alpha^2}{2}$. This way we obtain
\begin{align}
& \nu_-(\rho_- \varepsilon(\rho_-)- \rho_1 \varepsilon( \rho_1))+\nu_- 
\left(\rho_- \frac{v_{-2}^2}{2} - \rho_1 \frac{\beta^2+\ep_1+\ep_2}{2}\right)\nonumber\\
\leq & \left[(\rho_- \varepsilon(\rho_-)+ p(\rho_-)) v_{-2} - 
( \rho_1 \varepsilon( \rho_1)+ p(\rho_1)) \beta \right] 
+ \left( \rho_- \frac{v_{-2}^3}{2}- \rho_1 \beta \frac{\beta^2+\ep_1+\ep_2}{2}\right)\, .\label{eq:blbl1}
\end{align}
Using \eqref{eq:cont_left ss} we get $\nu_- = \frac{\rho_-v_{-2}-\rho_1\beta}{\rho_--\rho_1}$ and we plug this to the left hand side of \eqref{eq:blbl1}, while we multiply the right hand side by $\frac{\rho_--\rho_1}{\rho_--\rho_1}$ to get\footnote{Note that $\rho_- \neq \rho_1$, otherwise there is no solution to the studied system of identities and inequalities.}
\begin{align}
& \frac{\rho_1\rho_-}{\rho_--\rho_1}(\ep(\rho_1)-\ep(\rho_-))(\beta-v_{-2}) \nonumber\\ 
\leq & (p(\rho_-)v_{-2}-p(\rho_1)\beta) - \frac{\rho_1\rho_-}{2(\rho_--\rho_1)}\left((\beta-v_{-2})^2(\beta+v_{-2})+ (\beta-v_{-2})(\ep_1+\ep_2)\right)
\, .\label{eq:blbl2}
\end{align}
Combining \eqref{eq:cont_left ss} and \eqref{eq:mom_2_left ss} we get the following useful identity
\begin{equation}\label{eq:useful ss}
 (\rho_--\rho_1)(p(\rho_-)-p(\rho_1)) = \rho_-\rho_1(\beta-v_{-2})^2 + \ep_1\rho_1(\rho_--\rho_1).
\end{equation}
Finally inserting \eqref{eq:useful ss} into \eqref{eq:blbl2} we achieve the desired inequality \eqref{eq:E_left ss}.
\end{proof}

Note that according to Lemma \ref{l:admiss function} the expressions 
\begin{align}
 \left(p(\rho_-)+p(\rho_1)-2\rho_-\rho_1\frac{\ep(\rho_-)-\ep(\rho_1)}{\rho_--\rho_1}\right) \\
 \left(p(\rho_1)+p(\rho_+)-2\rho_1\rho_+\frac{\ep(\rho_1)-\ep(\rho_+)}{\rho_1-\rho_+}\right)
\end{align}
appearing on the left hand sides of \eqref{eq:E_left ss} and \eqref{eq:E_right ss} are both positive for $p(\rho) = \rho^\gamma$ with $\gamma \geq 1$.

For given data $\rho_\pm, v_{\pm 2}$ the system of relations \eqref{eq:cont_left ss}--\eqref{eq:E_right ss} consists of 4 equations and 4 inequalities for 6 unknowns $\nu_\pm, \rho_1, \beta, \ep_1, \ep_2$. Moreover $\ep_2$ appears only in the inequalities. Therefore we choose $\rho_1$ as a parameter and using the identities \eqref{eq:cont_left ss}--\eqref{eq:mom_2_right ss} we express $\nu_\pm, \beta$ and $\ep_1$ in terms of initial data and parameter $\rho_1$.

For simplicity we use the following notation for functions of initial data:
\begin{align}
 R &:= \rho_--\rho_+ \label{eq:R}\\
 A &:= \rho_-v_{-2}-\rho_+v_{+2} \label{eq:A}\\
 H &:= \rho_-v_{-2}^2-\rho_+v_{+2}^2 + p(\rho_-) - p(\rho_+). \label{eq:H}
\end{align}

Summing \eqref{eq:cont_left ss} and \eqref{eq:cont_right ss} we achieve
\begin{equation}\label{eq:num 1}
 \nu_- = \frac{A - \nu_+(\rho_1-\rho_+)}{\rho_--\rho_1}.
\end{equation}
Summing \eqref{eq:mom_2_left ss} and \eqref{eq:mom_2_right ss} we get
\begin{equation}\label{eq:blabla11}
 \nu_-^2(\rho_--\rho_1) + \nu_+^2(\rho_1-\rho_+) = H.
\end{equation}

Here we have to distinguish two cases. First let $R \neq 0$. Then \eqref{eq:blabla11} together with \eqref{eq:num 1} leads to two possible values of $\nu_+$:
\begin{equation}\label{eq:blabla12}
\nu_+ = \frac{A}{R} \pm \frac{1}{R}\sqrt{(A^2 - RH)\frac{\rho_1-\rho_-}{\rho_1-\rho_+}}.
\end{equation}
The proper sign is chosen in such a way that $\nu_- < \nu_+$. Observe that denoting $B:= A^2-RH $ and $u := v_{+2}-v_{-2}$ we have
\begin{equation}
B = \rho_-\rho_+u^2 - (\rho_--\rho_+)(p(\rho_-)-p(\rho_+)) 
\end{equation}
and thus the condition \eqref{eq:2shocks condition} implies that $B > 0$. According to Lemma \ref{l:Riemann selfsimilar solutions} the self-similar solution with this initial data consists of an admissible $1-$shock and an admissible $3-$shock with the density of the intermediate state $\rho_m > \max\{\rho_-,\rho_+\}$. This motivates us to try to find an admissible subsolution parametrized by $\rho_1 > \max\{\rho_-,\rho_+\}$. For this choice of $\rho_1$ we have
\begin{align}
 \nu_- = \frac{A}{R} - \frac{\sqrt{B}}{R}\sqrt{\frac{\rho_1-\rho_+}{\rho_1-\rho_-}} \label{eq:num}\\
 \nu_+ = \frac{A}{R} - \frac{\sqrt{B}}{R}\sqrt{\frac{\rho_1-\rho_-}{\rho_1-\rho_+}} \label{eq:nup}.
\end{align}
From \eqref{eq:cont_left ss} we can express $\beta$, although we will not need this expression in the future:
\begin{equation}\label{eq:beta}
 \beta = \frac{\rho_-v_{-2}}{\rho_1} - \frac{(\rho_--\rho_1)A}{R\rho_1} - \frac{\sqrt{B}}{R\rho_1}\sqrt{(\rho_1-\rho_-)(\rho_1-\rho_+)}.
\end{equation}
In the case $R > 0$, we use \eqref{eq:cont_left ss} and \eqref{eq:mom_2_left ss} to express $\ep_1$:
\begin{equation}\label{eq:ep1}
 \ep_1 = \left(\frac{\rho_+u}{R} + \frac{\sqrt{B}}{R}\sqrt{\frac{\rho_1-\rho_+}{\rho_1-\rho_-}}\right)^2\frac{\rho_-(\rho_1-\rho_-)}{\rho_1^2} - \frac{p(\rho_1)-p(\rho_-)}{\rho_1},
\end{equation}
while in the case $R < 0$, we rather use \eqref{eq:cont_right ss} and \eqref{eq:mom_2_right ss} to get
\begin{equation}\label{eq:ep10}
 \ep_1 = \left(\frac{\rho_-u}{R} + \frac{\sqrt{B}}{R}\sqrt{\frac{\rho_1-\rho_-}{\rho_1-\rho_+}}\right)^2\frac{\rho_+(\rho_1-\rho_+)}{\rho_1^2} - \frac{p(\rho_1)-p(\rho_+)}{\rho_1}.
\end{equation}

Now let $R = 0$, i.e. $\rho_-=\rho_+$. In this case similar procedure yields
\begin{align}
 \nu_- &= \frac{v_{-2}+v_{+2}}{2} - \frac{\rho_-\abs{u}}{2(\rho_1-\rho_-)} \label{eq:num2}\\
 \nu_+ &= \frac{v_{-2}+v_{+2}}{2} + \frac{\rho_-\abs{u}}{2(\rho_1-\rho_-)} \label{eq:nup2}\\
 \beta &= \frac{v_{-2}+v_{+2}}{2} \label{eq:beta2}\\
 \ep_1 &= \frac{\rho_-u^2}{4(\rho_1-\rho_-)} - \frac{p(\rho_1)-p(\rho_-)}{\rho_1} \label{eq:ep12}.
\end{align}

\begin{lemma}\label{l:ep1 decreasing}
 There exists a unique $\rho_{max} = \rho_{max}(\rho_-,\rho_+,u)$ such that 
 \begin{align}
  &\ep_1 > 0 \qquad \text{ for } \quad \rho_1 \in (\max\{\rho_-,\rho_+\},\rho_{max}) \\
  &\ep_1 < 0 \qquad \text{ for } \quad \rho_1 \in (\rho_{max},+\infty).
 \end{align}
  Moreover $\rho_{max} = \rho_m$, where $\rho_m$ is the density of the intermediate state of the self-similar solution emanating from initial data $(\rho_-,\rho_-v_{-2})$ on the left and $(\rho_+,\rho_+v_{+2})$ on the right given by Lemma \ref{l:Riemann selfsimilar solutions}. For fixed $\rho_-,\rho_+$ the value of $\rho_{max}$ grows asymptotically as $B^{\frac{1}{\gamma}}$, i.e. $u^{\frac{2}{\gamma}}$.
\end{lemma}
\begin{proof}
We start with the case $R = 0$, which is easy, because the function
\begin{equation}
 \ep_1(\rho_1) = \frac{\rho_-u^2}{4(\rho_1-\rho_-)} - \frac{p(\rho_1)-p(\rho_-)}{\rho_1}
\end{equation}
is obviously strictly decreasing with limits $+\infty$ as $\rho_1 \sil \rho_-$ and $-\infty$ as $\rho_1\sil+\infty$.

Now assume $R > 0$ and observe that the case $R < 0$ can be treated exactly in the same way just switching $\rho_-$ and $\rho_+$. We distinguish two cases. First assume that 
\begin{equation}\label{eq:ep1 as1}
u^2 \geq \frac{p(\rho_-)-p(\rho_+)}{\rho_+},
\end{equation}
which is equivalent to $\sqrt{B} \geq \rho_+\abs{u}$. Then observe that it is enough to show that function 
\begin{equation}\label{eq:ep1 1}
 \widetilde{\ep_1}(\rho_1) := \left(\sqrt{B}\sqrt{\frac{\rho_1-\rho_+}{\rho_1-\rho_-}} - \rho_+\abs{u}\right)^2\frac{\rho_1-\rho_-}{\rho_1}
\end{equation}
is decreasing on $(\rho_-,+\infty)$, because the term containing pressure in \eqref{eq:ep1} is obviously decreasing. We rewrite \eqref{eq:ep1 1} further to
\begin{equation}\label{eq:ep1 2}
 \widetilde{\ep_1}(\rho_1) := \left(\sqrt{B}\sqrt{1-\frac{\rho_+}{\rho_1}}-\rho_+\abs{u}\sqrt{1-\frac{\rho_-}{\rho_1}}\right)^2
\end{equation}
and study the function
\begin{equation}\label{eq:ep1 3}
 f(\rho_1) := \sqrt{B}\sqrt{1-\frac{\rho_+}{\rho_1}}-\rho_+\abs{u}\sqrt{1-\frac{\rho_-}{\rho_1}}.
\end{equation}
Note that $f(\rho_1) > 0$ for all $\rho_1 > \rho_-$ due to assumption \eqref{eq:ep1 as1}. A standard calculation yields
\begin{equation}\label{eq:ep1 4}
 f'(\rho_1) := \frac{\rho_1^{-\frac{3}{2}}}{2}\left(\frac{\sqrt{B}\rho_+}{\sqrt{\rho_1-\rho_+}} - \frac{\rho_+\rho_-\abs{u}}{\sqrt{\rho_1-\rho_-}}\right).
\end{equation}
Observing that 
\begin{equation}\label{eq:ep1 5}
 \rho_-\rho_+\abs{u}^2 - (\rho_--\rho_+)(p(\rho_-)-p(\rho_+)) = B < \rho_-^2\abs{u}^2
\end{equation}
we conclude that $f'(\rho_1)$ is negative on $(\rho_-,+\infty)$, thus $f(\rho_1)$ is decreasing on this interval. Therefore also $\widetilde{\ep_1}(\rho_1)$ and $\ep_1(\rho_1)$ are decreasing on $(\rho_-,+\infty)$ with $\lim_{\rho_1 \sil +\infty}\ep(\rho_1) < 0$.

Next let 
\begin{equation}\label{eq:ep1 as2}
u^2 < \frac{p(\rho_-)-p(\rho_+)}{\rho_+}
\end{equation}
and thus $\sqrt{B} < \rho_+\abs{u}$. Then there exist a finite $\widetilde{\rho}$ such that 
\begin{equation}\label{eq:ep1 6}
 \sqrt{B}\sqrt{1-\frac{\rho_+}{\widetilde{\rho}}} = \rho_+\abs{u}\sqrt{1-\frac{\rho_-}{\widetilde{\rho}}}.
\end{equation}
On interval $(\rho_-,\widetilde{\rho})$ we can argue in the same way as above, function $\widetilde{\ep_1}$ is on this interval clearly decreasing. Moreover obviously $\ep_1(\widetilde{\rho}) = -\frac{p(\widetilde{\rho})-p(\rho_-)}{\widetilde{\rho}} < 0$. Therefore it is enough now to prove that $\ep_1(\rho)$ stays negative on $(\widetilde{\rho},+\infty)$. In other words we want to prove that it holds
\begin{equation}\label{eq:ep1 7}
 \left(\rho_+\abs{u} - \sqrt{B}\sqrt{\frac{\rho_1-\rho_+}{\rho_1-\rho_-}}\right)^2 < \frac{p(\rho_1)-p(\rho_-)}{\rho_1 - \rho_-}\frac{\rho_1R^2}{\rho_-}
\end{equation}
on the interval $(\widetilde{\rho},+\infty)$. We claim that it always holds
\begin{equation}\label{eq:ep1 8}
 (\rho_+\abs{u} - \sqrt{B})^2 = \left(\rho_+\abs{u} - \sqrt{\rho_-\rho_+\abs{u}^2 - (\rho_--\rho_+)(p(\rho_-)-p(\rho_+))}\right)^2 < p'(\rho_-)R^2.
\end{equation}
Inequality \eqref{eq:ep1 8} is proved by taking supremum over all possible $\abs{u}$ on the left hand side. Note that possible values of $\abs{u}$ form the interval $\Big(\sqrt{\frac{(\rho_--\rho_+)(p(\rho_-)-p(\rho_+))}{\rho_-\rho_+}}, \sqrt{\frac{p(\rho_-)-p(\rho_+)}{\rho_+}}\Big)$. It is not difficult to see that the supremum is achieved in point $\abs{u} = \sqrt{\frac{(\rho_--\rho_+)(p(\rho_-)-p(\rho_+))}{\rho_-\rho_+}}$ and its value is 
\begin{equation}\label{eq:ep1 9}
 \sup_{\abs{u}}\, (\rho_+\abs{u} - \sqrt{B})^2 = \frac{\rho_+}{\rho_-}(p(\rho_-)-p(\rho_+))(\rho_--\rho_+).
\end{equation}
Arguing that $p(\rho_-)-p(\rho_+) = p'(\xi)(\rho_--\rho_+)$ for some $\xi \in (\rho_+,\rho_-)$ we achieve\footnote{Note that $p'(\rho)$ is nondecreasing for $p(\rho) = \rho^\gamma$ with $\gamma \geq 1$.}
\begin{equation}\label{eq:ep1 10}
 (\rho_+\abs{u} - \sqrt{B})^2  \leq \sup_{\abs{u}}\, (\rho_+\abs{u} - \sqrt{B})^2 = \frac{\rho_+}{\rho_-}p'(\xi)R^2 < p'(\rho_-)R^2,
\end{equation}
which proves \eqref{eq:ep1 8}. The proof of \eqref{eq:ep1 7} is now completed by the following chain of inequalities:
\begin{equation}\label{eq:ep1 11}
 \left(\rho_+\abs{u} - \sqrt{B}\sqrt{\frac{\rho_1-\rho_+}{\rho_1-\rho_-}}\right)^2 < (\rho_+\abs{u} - \sqrt{B})^2 < p'(\rho_-)R^2 < \frac{p(\rho_1)-p(\rho_-)}{\rho_1 - \rho_-}\frac{\rho_1R^2}{\rho_-}.
\end{equation}

Concerning the second statement of the Lemma, it is enough to observe that equations \eqref{eq:cont_left ss}--\eqref{eq:mom_2_right ss} with $\ep_1 = 0$ are exactly the Rankine--Hugoniot shock conditions which are satisfied by the self-similar solution. As is shown in the following Lemma \ref{l:positivity conditions} and formulas \eqref{eq:ep2 2} and \eqref{eq:ep2 4}, the admissibility conditions \eqref{eq:E_left ss}--\eqref{eq:E_right ss} together with \eqref{eq:sub_2 ss} yield $v_{+2} < \beta < v_{-2}$, i.e. the same conclusion as the admissibility conditions for the self-similar solution.

The asymptotic growth of $\rho_{max}$ with respect to $B \sim u^2 \sil \infty$ is an easy observation.
\end{proof}

To study consequences of the admissibility inequalities \eqref{eq:E_left ss}--\eqref{eq:E_right ss} we need to know the signs of $v_{-2} - \nu_-$ and $\nu_+ - v_{+2}$.
\begin{lemma}\label{l:positivity conditions}
 For any $u < -\sqrt{\frac{(\rho_--\rho_+)(p(\rho_-)-p(\rho_+))}{\rho_-\rho_+}}$ and any $\rho_1 \in (\max\{\rho_-,\rho_+\},\rho_{max})$ it holds
 \begin{align}
  v_{-2}-\nu_- &> 0 \label{eq:cond1}\\
  \nu_+ - v_{+2} &>0. \label{eq:cond2}
 \end{align}
\end{lemma}
\begin{proof}
In the case $R = 0$ relations \eqref{eq:cond1}--\eqref{eq:cond2} follow directly from \eqref{eq:num2} and \eqref{eq:nup2}.

Now we prove Lemma \ref{l:positivity conditions} in the case $R > 0$ and again claim that the case $R < 0$ can be treated by the same arguments. The inequality \eqref{eq:cond1} follows directly from the proof of Lemma \ref{l:ep1 decreasing}. We have proved there that 
 \begin{equation}
  v_{-2}-\nu_-(\rho_1) = \frac{\rho_+u}{R} + \frac{\sqrt{B}}{R}\sqrt{\frac{\rho_1-\rho_+}{\rho_1-\rho_-}}
 \end{equation}
 is in the case $u^2 \geq \frac{p(\rho_-)-p(\rho_+)}{\rho_+}$ positive for all $\rho_1 > \rho_-$, while in the case $u^2 < \frac{p(\rho_-)-p(\rho_+)}{\rho_+}$ it is positive on the interval $(\rho_-,\widetilde{\rho})$ with $\widetilde{\rho} > \rho_{\max}$.
 
 To prove the inequality \eqref{eq:cond2} we proceed as follows:
 \begin{align}
 &\nu_+(\rho_1) - v_{+2} = \frac{\rho_-\abs{u}}{R} - \frac{\sqrt{B}}{R}\sqrt{\frac{\rho_1-\rho_-}{\rho_1-\rho_+}} > \frac{1}{R}\left(\rho_-\abs{u} - \sqrt{B}\right) \nonumber \\
 = &\frac{1}{R}\left(\rho_-\abs{u} - \sqrt{\rho_-\rho_+\abs{u}^2 - (\rho_--\rho_+)(p(\rho_-)-p(\rho_+))}\right) \nonumber \\
 > &\frac{1}{R}\left(\rho_-\abs{u} - \sqrt{\rho_-\rho_+\abs{u}^2}\right) > 0.
 \end{align}
\end{proof}

We now rewrite further the admissibility inequalities \eqref{eq:E_left ss}--\eqref{eq:E_right ss}. First observe that from \eqref{eq:cont_left ss} and \eqref{eq:cont_right ss} we have
\begin{align}
 \beta - \nu_- &= \frac{\rho_-}{\rho_1}\left(v_{-2}-\nu_-\right) \label{eq:ep2 1}\\
 v_{-2} - \beta &= \frac{\rho_1-\rho_-}{\rho_1}\left(v_{-2}-\nu_-\right) \label{eq:ep2 2}\\
 \nu_+ - \beta &= \frac{\rho_+}{\rho_1}\left(\nu_+ - v_{+2}\right) \label{eq:ep2 3}\\
 \beta - v_{+2} &= \frac{\rho_1-\rho_+}{\rho_1}\left(\nu_+-v_{+2}\right) \label{eq:ep2 4}
\end{align}
and thus all the terms on the left hand sides of \eqref{eq:ep2 1}--\eqref{eq:ep2 4} are positive on the interval $(\max\{\rho_-,\rho_+\},\rho_{max})$. Let us denote
\begin{equation}\label{eq:P def}
 P(r,s) := p(r)+p(s)-2rs\frac{\ep(r)-\ep(s)}{r-s}
\end{equation}
and recall that for $p(\rho) = \rho^{\gamma}$, $\gamma \geq 1$ it holds $P(r,s) > 0$ for $r \neq s$ due to Lemma \ref{l:admiss function}. We now rewrite the admissibility condition on the left interface \eqref{eq:E_left ss} as follows
\begin{equation}\label{eq:E_left 3}
(\nu_--v_{-2})\frac{\rho_1-\rho_-}{\rho_1}P(\rho_-,\rho_1) \leq \ep_1\rho_1(v_{-2}+\beta) + (\ep_1+\ep_2)\rho_-(\nu_--v_{-2}).
\end{equation}
By an easy calculation we get
\begin{equation}\label{eq:E_left 4}
(\ep_1+\ep_2)\rho_-(v_{-2}-\nu_-) \leq \ep_1\rho_1(v_{-2}+\beta) + (v_{-2}-\nu_-)\frac{\rho_1-\rho_-}{\rho_1}P(\rho_-,\rho_1)
\end{equation}
and consequently
\begin{equation}\label{eq:E left final}
 \ep_2 < \frac{P(\rho_-,\rho_1)(\rho_1-\rho_-)}{\rho_1\rho_-} + \ep_1\frac{\rho_1(v_{-2}+\nu_-)}{\rho_-(v_{-2}-\nu_-)}.
\end{equation}
Treating the admissibility condition on the right interface \eqref{eq:E_right ss} in a similar way we achieve
\begin{equation}\label{eq:E right final}
 \ep_2 < \frac{P(\rho_+,\rho_1)(\rho_1-\rho_+)}{\rho_1\rho_+} - \ep_1\frac{\rho_1(\nu_++v_{+2})}{\rho_-(\nu_+-v_{+2})}.
\end{equation}

The proof of Theorem \ref{t:main0} is finished by observing that in the point $\rho_1 = \rho_{max}$ it holds $\ep_1(\rho_1) = 0$ and thus the right hand sides of \eqref{eq:E left final} and \eqref{eq:E right final} are both strictly positive. Thus, by a simple continuity argument, we conclude that for any initial Riemann data satisfying \eqref{eq:2shocks condition} there exist (in fact even infinitely many) admissible subsolutions parametrized by $\rho_1$ belonging to some left neighborhood of $\rho_{max}$.

\section{Proof of Theorem \ref{t:main}}\label{s:Dis}

In this section we prove Theorem \ref{t:main}. First we recall the definitions of energy and dissipation rate of a solution from Section \ref{s:entropy}.
\begin{align}
&E_L[\rho,v](t) = \int_{(-L,L)^2} \left(\rho\ep(\rho) + \rho\frac{\abs{v}^2}{2}\right)\d x \label{eq:energy L 2}\\
&D_L[\rho,v](t) = \frac{\d_{+} E_L[\rho,v](t)}{\d t}. \label{eq:dissipation rate L 2}
\end{align}
Assume from now on for simplicity that 
\begin{equation}\label{eq:alpha 0}
 v_{-1} = v_{+1} = \alpha = 0.
\end{equation}
The value of the dissipation rate $D_L[\rho_c,v_c](t)$ has a specific form when the solution $(\rho_c,v_c)$ is a self-similar solution consisting of two shocks of speeds $\nu_-$ and $\nu_+$. Denoting the intermediate state $(\rho_m,\rho_mv_m) = (\rho_m,(0,\rho_mv_{m2}))$ and introducing the notation
\begin{align}
 E_\pm &:= \rho_\pm\ep(\rho_\pm) + \rho_\pm\frac{v_\pm^2}{2} \label{eq:E pm} \\
 E_m &:= \rho_m\ep(\rho_m) + \rho_m\frac{v_{m2}^2}{2}
\end{align}
we obtain
\begin{equation}\label{eq:Dis rate for 2 shocks}
 D_L[\rho_c,v_c](t) = 2L\left(\nu_-(E_--E_m) + \nu_+(E_m-E_+)\right)
\end{equation}
at least for $t \leq T^*$ with some $T^*$ depending on $L^*$.

Now let us consider a solution $(\rho_n,v_n)$ with the same initial data constructed from an admissible subsolution using Proposition \ref{p:subs}. Although $v_n$ is not constant in $P_1$ we still have that $\abs{v_n}^2\bm{1}_{P_1} = C = \beta^2 + \ep_1+\ep_2$, in particular
\begin{equation}
 E_1 = \rho_1\ep(\rho_1) + \rho_1\frac{\beta^2+\ep_1+\ep_2}{2}
\end{equation}
is constant in $P_1$. The dissipation rate for all solutions $(\rho_n,v_n)$ constructed from a given subsolution with intermediate state $(\rho_1,\rho_1v_1) = (\rho_1,(0,\rho_1\beta))$ is thus given by a similar expression as \eqref{eq:Dis rate for 2 shocks}, more precisely it holds
\begin{equation}\label{eq:Dis rate for subsol}
 D_L[\rho_n,v_n](t) = 2L\left(\nu_-(E_--E_1) + \nu_+(E_1-E_+)\right)
\end{equation}
again at least for sufficiently small $t$.

Let us now therefore study further properties of the function 
\begin{equation}\label{eq:def f}
 f(\rho_1) := \nu_-(\rho_1)(E_- - E_1(\rho_1)) + \nu_+(\rho_1)(E_1(\rho_1) - E_+) = \frac{D_L[\rho_n,v_n](t)}{2L}
\end{equation}
in a case $p(\rho) = \rho^\gamma$, $\gamma \geq 1$. First observe that our goal is to make $f(\rho_1)$ small. Considering the dependence of $f(\rho_1)$ on $\ep_2$ we easily see that the smallest possible value of $f(\rho_1)$ is achieved by taking $\ep_2 = 0$.  In fact it is easy to see that it holds
\begin{equation}\label{eq:gg1}
 \lim_{\rho_1 \sil \rho_{max}} \lim_{\ep_2 \sil 0} 2Lf(\rho_1) = D_L[\rho_c,v_c](t)
\end{equation}
for sufficiently small $t$. Indeed, using Lemma \ref{l:ep1 decreasing} we easily see that $\rho_{max} = \rho_m$ and for $\ep_1 = \ep_2 = 0$ we get also $\beta = v_{m2}$ and $\nu_\pm(\rho_{max})$ are exactly the shock speeds of the self-similar solution. We conclude therefore that Theorem \ref{t:main} is a direct consequence of the following Lemma.
\begin{lemma}\label{l:f increasing}
 Let $1\leq \gamma < 3$. There exist initial data $\rho_\pm$, $v_{\pm2}$ for which the function $f(\rho_1)$ defined in \eqref{eq:def f} is increasing in the neighborhood of $\rho_{max}$.
\end{lemma}
\begin{remark}
 We find such initial data by analysing the case $R > 0$, nevertheless by the same arguments we could find initial data satisfying Lemma \ref{l:f increasing} also with $R = 0$ and $R < 0$.
 \end{remark}
\begin{proof}
 Let $R > 0$. Denote 
 \begin{equation}
Q(\rho) := 2\rho\ep(\rho)-p(\rho).
 \end{equation}
 and observe that $Q(\rho) = \frac{3-\gamma}{\gamma-1}\rho^\gamma$ in the case $\gamma > 1$, whereas $Q(\rho) = 2\rho\log\rho-\rho$ for $\gamma = 1$. Using \eqref{eq:mom_2_left ss} and \eqref{eq:mom_2_right ss} we achieve
 \begin{align}
  E_- - E_1(\rho_1) &= \frac{1}{2}\left(Q(\rho_-)-Q(\rho_1)+\nu_-^2(\rho_1)(\rho_--\rho_1)\right) \label{eq:gg2}\\
  E_1(\rho_1) - E_+ &= \frac{1}{2}\left(Q(\rho_1)-Q(\rho_+)+\nu_+^2(\rho_1)(\rho_1-\rho_+)\right) \label{eq:gg3}
 \end{align}
and thus plugging in the expressions \eqref{eq:num}, \eqref{eq:nup} for $\nu_-$ and $\nu_+$ we get
\begin{align}
 f(\rho_1) &= \frac{1}{2R^3}\left[\left(A - \sqrt{B}\sqrt{\frac{\rho_1-\rho_+}{\rho_1-\rho_-}}\right)^3(\rho_--\rho_1) + \left(A - \sqrt{B}\sqrt{\frac{\rho_1-\rho_-}{\rho_1-\rho_+}}\right)^3(\rho_1-\rho_+)\right] \nonumber\\
 &+ \frac{1}{2R}\left[\left(A - \sqrt{B}\sqrt{\frac{\rho_1-\rho_+}{\rho_1-\rho_-}}\right)(Q(\rho_-)-Q(\rho_1))\right. \nonumber\\
 &+ \qquad  \left.\left(A - \sqrt{B}\sqrt{\frac{\rho_1-\rho_-}{\rho_1-\rho_+}}\right)(Q(\rho_1)-Q(\rho_+))\right]. \label{eq:f 1}
\end{align}
By an easy calculation we find out that some terms are in fact constant. Denoting
\begin{equation}\label{eq:gg4}
 C_0 := \frac{A^3-3AB}{2R^2} + \frac{A(Q(\rho_-)-Q(\rho_+))}{2R}
\end{equation}
we have
\begin{align}\label{eq:gg5}
 f(\rho_1) &= C_0 + \frac{B^{\frac{3}{2}}}{2R^2}\frac{2\rho_1-\rho_--\rho_+}{\sqrt{(\rho_1-\rho_-)(\rho_1-\rho_+)}} \nonumber\\
 &+ \frac{\sqrt{B}}{2R}\frac{RQ(\rho_1) - \rho_1(Q(\rho_-)-Q(\rho_+)) + \rho_+Q(\rho_-)-\rho_-Q(\rho_+)}{\sqrt{(\rho_1-\rho_-)(\rho_1-\rho_+)}}
\end{align}
and therefore we study further the function 
\begin{equation}\label{eq:gg6}
 g(\rho_1) = \frac{\frac{B}{R^2}(2\rho_1-\rho_--\rho_+) + Q(\rho_1) - \rho_1\frac{Q(\rho_-)-Q(\rho_+)}{R} + \frac{\rho_+Q(\rho_-)-\rho_-Q(\rho_+)}{R}}{\sqrt{(\rho_1-\rho_-)(\rho_1-\rho_+)}}.
\end{equation}
Straightforward calculation yields
\begin{align}
g'(\rho_1) &= \frac{1}{((\rho_1-\rho_-)(\rho_1-\rho_+))^\frac{3}{2}} \nonumber\\ 
&\cdot\left[\left(Q'(\rho_1) + \frac{2B}{R^2} - \frac{Q(\rho_-)-Q(\rho_+)}{R}\right)(\rho_1-\rho_-)(\rho_1-\rho_+)\right. \nonumber\\ 
&\quad-\left. \frac{2\rho_1-\rho_--\rho_+}{2}\left(\frac{B}{R^2}(2\rho_1-\rho_--\rho_+) \right.\right. \nonumber\\
&\qquad +\left.\left. Q(\rho_1) - \rho_1\frac{Q(\rho_-)-Q(\rho_+)}{R} + \frac{\rho_+Q(\rho_-)-\rho_-Q(\rho_+)}{R}\right)\right].\label{eq:gg7}
\end{align}

\begin{proposition}\label{p:g}
 For any $1 \leq \gamma < 3$ and any couple of densities $\rho_- > \rho_+$ there exists a unique local minimum $\overline{\rho} > \rho_-$ of the function $g(\rho_1)$. For fixed $\gamma, \rho_-,\rho_+$ the value of $\overline{\rho}$ grows asymptotically as $B^{\frac{1}{\gamma+1}}$.
\end{proposition}
Assuming Proposition \ref{p:g} holds we now conclude the proof of Lemma \ref{l:f increasing} easily comparing the asymptotic growth of $\rho_{max}$ stated in Lemma \ref{l:ep1 decreasing} and the asymptotic growth of $\overline{\rho}$ stated in Proposition \ref{p:g}. Since $\overline{\rho}$ grows slower than $\rho_{max}$, the proof of Lemma \ref{l:f increasing} and thus Theorem \ref{t:main} is finished by taking $\abs{u}$ large enough so that $\overline{\rho} < \rho_{max}$.
\end{proof}

\begin{proof}[Proposition \ref{p:g}]
 From \eqref{eq:gg7} we derive the equation satisfied by any critical point of $g(\rho_1)$.
 \begin{align}
 &2\left(Q'(\rho_1)-\frac{Q(\rho_-)-Q(\rho_+)}{R}\right)(\rho_1-\rho_-)(\rho_1-\rho_+) \nonumber\\ 
 &\quad - (2\rho_1-\rho_--\rho_+)\left(Q(\rho_1) - \rho_1\frac{Q(\rho_-)-Q(\rho_+)}{R} + \frac{\rho_+Q(\rho_-)-\rho_-Q(\rho_+)}{R}\right) = B. \label{eq:ggg0}
 \end{align}
Denote $z(\rho_1)$ the function on the left hand side of \eqref{eq:ggg0} and observe that $z(\rho_-) = 0$. Moreover 
\begin{align}
 z'(\rho_1) &= (2\rho_1-\rho_--\rho_+)\left(Q'(\rho_1)-\frac{Q(\rho_-)-Q(\rho_+)}{R}\right) + 2(\rho_1-\rho_-)(\rho_1-\rho_+)Q''(\rho_1) \\ \nonumber
 &-2\left(Q(\rho_1) - \rho_1\frac{Q(\rho_-)-Q(\rho_+)}{R} + \frac{\rho_+Q(\rho_-)-\rho_-Q(\rho_+)}{R}\right). \label{eq:gggn1}
\end{align}
We claim that $z'(\rho_-) > 0$. Indeed easy calculation yields in the case $\gamma > 1$
\begin{equation}\label{eq:ggg3}
 z'(\rho_-) = \frac{3-\gamma}{\gamma-1}\left((\gamma-1)\rho_-^{\gamma}-\gamma\rho_-^{\gamma-1}\rho_++\rho_+^{\gamma}\right)
\end{equation}
and denoting $w(\rho_-)=\rho_-^{\gamma}$ we have for some $\xi \in (\rho_+,\rho_-)$
\begin{equation}\label{eq:ggg4}
 \gamma\rho_-^{\gamma-1} = w'(\rho_-) > w'(\xi) = \frac{\rho_-^{\gamma}-\rho_+^\gamma}{\rho_--\rho_+}
\end{equation}
which together with \eqref{eq:ggg3} proves that $z'(\rho_-) > 0$. 

In the case $\gamma = 1$ we plug in the expression $Q(\rho) = 2\rho\log\rho - \rho$ and we can simplify the resulting expression up to 
\begin{equation}\label{eq:ggg3n}
 z'(\rho_-) = \rho_- - \rho_+ - \rho_+\log\frac{\rho_-}{\rho_+}.
\end{equation}
Dividing this by $\rho_+$ and denoting $r=\frac{\rho_-}{\rho_+}$, the positivity of $z'(\rho_-)$ follows from the fact that function $r-1-\log r$ is positive for $r > 1$.

Next we have
\begin{align}\label{eq:ggg5}
 z''(\rho_1) &= 3(2\rho_1-\rho_--\rho_+)Q''(\rho_1) + 2(\rho_1-\rho_-)(\rho_1-\rho_+)Q'''(\rho_1) \\ \nonumber 
 &= (3-\gamma)\gamma\rho_1^{\gamma-3}\left((2\gamma+2)\rho_1^2-(2\gamma-1)\rho_1(\rho_-+\rho_+)+(2\gamma-4)\rho_-\rho_+\right) \\ \nonumber
 &= (3-\gamma)\gamma\rho_1^{\gamma-3} z_0(\rho_1). \label{eq:ggg5}
\end{align}
for $\gamma = 1$ as well as for $\gamma > 1$. Again we easily compute $z_0(\rho_-) = 3\rho_-(\rho_--\rho_+) > 0$. The local minimum of the quadratic function $z_0(\rho_1)$ is in point
\begin{equation}\label{eq:ggg6}
 \rho_0 = \frac{(2\gamma-1)(\rho_-+\rho_+)}{4(\gamma+1)}
\end{equation}
and it is not difficult to check that $\rho_0 < \rho_-$. We now conclude that on the interval $(\rho_-,+\infty)$ the function $z_0(\rho_1)$ is positive and therefore $z''(\rho_1)$ is there also positive. This means that $z'(\rho_1)$ is increasing on $(\rho_-,+\infty)$ and since we already know that $z'(\rho_-) > 0$ we have in particular that $z'$ is positive on $(\rho_-,+\infty)$ and therefore $z(\rho_1)$ is strictly increasing on $(\rho_-,+\infty)$. The equation \eqref{eq:ggg0} has therefore for $1 \leq \gamma < 3$ a unique solution $\overline{\rho}$ for any positive $B$. Since asymptotically $z(\rho_1) \sim \rho_1^{\gamma+1}$ we also conclude that $\overline{\rho} \sim B^{\frac{1}{\gamma+1}}$.
\end{proof}

\section{Concluding remarks}\label{s:per}

The problem of infinite domain and therefore infinite energy might be solved also restricting to a periodic domain. Note that the self-similar solution does not depend on $x_1$ and therefore it can be seen as periodic in $x_1$ with any period we like. The nonstandard solutions on the other hand depend on $x_1$. However the subsolution which provides the existence of infinitely many solutions does not. Therefore from this subsolution we can in a similar way construct infinitely many periodic (in $x_1$) solutions with a given period.

The situation in the direction $x_2$ is a little more complicated. We will use the following easy Proposition to avoid possible problems with vacuum. 
\begin{proposition}\label{p:vacuum}
Let $\rho_\pm > 0$ and $v_{\pm 2}$ be given. Then there exists finite $N > 0$ and a family of states $(\rho^k,\rho^kv^k_2)$, $k=1,...,N+1$ such that 
\begin{itemize}
 \item $(\rho^1,\rho^1v^1_2) = (\rho_+,\rho_+v_{+2})$
 \item $(\rho^{N+1},\rho^{N+1}v^{N+1}_2) = (\rho_-,\rho_-v_{-2})$
 \item For all $k=1,...,N$ the solution of the Riemann problem \eqref{eq:reduced system} with initial data 
 \begin{equation}\label{eq:Rdata m3}
(\rho^0 (x), m^0 (x)) := \left\{
\begin{array}{ll}
(\rho^{k},\rho^kv^{k}_2) \quad & \mbox{if $x_2<0$}\\ \\
(\rho^{k+1},\rho^{k+1}v^{k+1}_2)  & \mbox{if $x_2>0$.} 
\end{array}\right. 
\end{equation}
does not contain an intermediate vacuum region.
\end{itemize}
\end{proposition} 
\begin{proof}
 According to Lemma \ref{l:Riemann selfsimilar solutions} we immediately conclude that in the case
 \begin{equation}\label{eq:hh1}
  v_{-2} - v_{+2} < \int_0^{\rho_+} \frac{\sqrt{p'(\tau)}}{\tau} \d\tau + \int_0^{\rho_-} \frac{\sqrt{p'(\tau)}}{\tau} \d\tau
 \end{equation}
 we can set $N = 1$. If \eqref{eq:hh1} does not hold we set\footnote{Here $[x]$ denotes the integer part of the real number $x$.}
 \begin{equation}\label{eq:hh2}
  N := \left[\frac{v_{-2}-v_{+2}}{\int_0^{\rho_+} \frac{\sqrt{p'(\tau)}}{\tau} \d\tau}\right] + 1
 \end{equation}
 and for $k=2,...,N$ define
 \begin{align}
  \rho^k &:= \rho_+ \label{eq:hh3}\\ 
  v^k_2 &:= v_{+2} + (k-1)\left(\int_0^{\rho_+} \frac{\sqrt{p'(\tau)}}{\tau} \d\tau\right). \label{eq:hh30}
 \end{align}
 It is now easy to check that none of the solutions to Riemann problems with initial data $(\rho^k,\rho^kv^k_2)$ on the left and $(\rho^{k+1},\rho^{k+1}v^{k+1}_2)$ on the right contains an intermediate vacuum region.
\end{proof} 

For given initial data $\rho_\pm > 0, v_{\pm 2}$ let $N > 0$ and $(\rho^k,\rho^kv^k_2)$ be given by Proposition \ref{p:vacuum}. We construct the periodic problem as follows. Consider the initial data:
\begin{align}
 \rho^0 &= \rho^{N+1} \qquad \text{ for } x_2 \in \left(0,\frac{1}{2N+2}\right) \cup \left(\frac{2N+1}{2N+2},1\right)\label{eq:IC 1} \\
 \rho^0 &= \rho^k \qquad \text{ for } x_2 \in \left(\frac{2k-1}{2N+2},\frac{2k+1}{2N+2}\right) \qquad \text{ for } k=1,...,N
\end{align}
and similarly
\begin{align}
 v^0 &= v^{N+1} = (0,v^{N+1}_2) \qquad \text{ for } x_2 \in \left(0,\frac{1}{2N+2}\right) \cup \left(\frac{2N+1}{2N+2},1\right) \\
 v^0 &= v^k = (0,v^k_2) \qquad \text{ for } x_2 \in \left(\frac{2k-1}{2N+2},\frac{2k+1}{2N+2}\right) \qquad \text{ for } k=1,...,N\label{eq:IC 4}
\end{align} 
and we will consider solutions which are spatially periodic with period $(0,1)^2$ and local in time. By working in this periodic setting we in particular ensure that the dissipation rate $D[\rho,v](t)$ is 
nonpositive for any admissible solution $(\rho,v)$.

We construct locally in time two types of solutions. First consider a classical self-similar solution $(\rho_c,v_c)$ to a problem with initial data \eqref{eq:IC 1}--\eqref{eq:IC 4} keeping the assumption that the initial data satisfy
\begin{equation}
 u = v_{+2}-v_{-2} < -\sqrt{\frac{(\rho_--\rho_+)(p(\rho_-)-p(\rho_+))}{\rho_-\rho_+}}.
\end{equation}
We see that such solution in particular consists of an admissible $1-$shock and an admissible $3-$shock emanating from line $x_2 = \frac{1}{2N+2}$. 

Then consider (infinitely many) admissible nonstandard solutions $(\rho_n,v_n)$ constructed from an admissible subsolution consisting of a piecewice constant values of $\rho$ and $\abs{v}^2$ in regions $P_-$, $P_1$ and $P_+$ emanating from line $x_2 = \frac{1}{2N+2}$ and a self-similar part (consisting of admissible shocks and rarefaction waves) emanating from lines $x_2 = \frac{2k+1}{2N+2}$, $k=1,...,N$. The self-similar part of such solutions $(\rho_n,v_n)$ is the same as the appropriate part of the self-similar solution $(\rho_c,v_c)$.

By the analysis in the previous section we conclude, that there exist  initial data \eqref{eq:IC 1}--\eqref{eq:IC 4} for which there exist (infinitely many) nonstandard periodic admissible solutions $(\rho_n,v_n)$ which dissipate locally in time more energy than the self-similar solution $(\rho_c,v_c)$, i.e.
\begin{equation}
 D[\rho_n,v_n](t) < D[\rho_c,v_c](t)
\end{equation}
 for all $0 < t < T^*$ with $T^* > 0$ given by the initial data is the time of first interaction between waves and shocks emanating from different lines. 
Here $D[\rho,v](t)$ is defined as in \eqref{eq:dissipation rate0} keeping in mind that the energy $E[\rho,v](t)$ is defined as an integral over $(0,1)^2$.

\subsection*{Acknowledgement} The authors are very thankful to Camillo De Lellis for suggesting them the problem and for many enlightening discussions.

\end{document}